\documentclass[11pt]{amsart}
\usepackage{geometry}
\usepackage{amscd,amssymb,verbatim,xcolor}
\usepackage{longtable, multirow,relsize}
\usepackage{soul,cancel}
\usepackage{apptools}
\AtAppendix{\counterwithin{theorem}{section}}
\date{\today}

\newcommand\al{\alpha}

\newcommand\bb[1]{{\mathbb     #1}}
\newcommand\C[1]{{\mathcal #1 }}
\newcommand\fk[1]{\mathfrak{#1}}
\newcommand\ovl[1]{\overline{#1}}

\newcommand\CO{{\mathcal O}}
\newcommand\cCO{{\ovl{\mathcal O}}}

\newcommand\bC{{\mathbb C}}

\newcommand\bZ{{\mathbb Z}}

\newcommand\la{{\lambda}}
\newcommand\ep{{\epsilon}}
\newcommand\sig{{\sigma}}

\newcommand\fg{{\mathfrak g}}

\newcommand\ie{{i.e. ~}}
\newcommand\eg{{e.g. ~}}

\newtheorem{theorem}{Theorem}[section]

\newtheorem{corollary}[theorem]{Corollary}

\newtheorem{definition}[theorem]{Definition}
\newtheorem{example}[theorem]{Example}
\newtheorem{lemma}[theorem]{Lemma}
\newtheorem{proposition}[theorem]{Proposition}
\newtheorem{remark}[theorem]{Remark}

\newcommand\Ind{{\operatorname{Ind}}}

\begin{document}
\title{Admissible modules and normality of classical nilpotent orbits I}
\author{Dan Barbasch}
\author{Kayue Daniel Wong}

\address[Barbasch]{Department of Mathematics, Cornell University, Ithaca, NY 14853,
U.S.A.}
\email{barbasch@math.cornell.edu}

\address[Wong]{School of Science and Engineering, The Chinese University of Hong Kong, Shenzhen,
Guangdong 518172, P. R. China}
\email{kayue.wong@gmail.com}

\begin{abstract}
In the case of complex symplectic and orthogonal groups, we
find $(\fk g, K)-$modules with the property that their $K-$structure
matches the structure of regular functions on the closures of
nilpotent orbits. This establishes a version of the \textit{Orbit
  Method} of Kirrilov-Kostant-Souriau as proposed by Vogan. In the
process we give another proof of the classification of nilpotent
orbits with normal closure in the Lie algebra of a classical group
first established by Kraft-Procesi. 
\end{abstract}

\maketitle
\setcounter{tocdepth}{1}

%%%%%%%%%%%%%%%%%%%
%%%%%%%%%%%%%%%%%%%
%%%SECTION 1%%%%%%%
%%%%%%%%%%%%%%%%%%%
%%%%%%%%%%%%%%%%%%%
\section{Introduction}
\label{sec:intro}
% \subsection{Introduction}

Normality of an algebraic variety is an essential property in
algebraic geometry. Much insight can be gained from the case of a
variety with a group action.  In their seminal work, Kraft and
Procesi, \cite{KP1} and \cite{KP2}, classify orbits with normal
closure for the action of a classical
group, $GL(n,\bb C)$, $Sp(2n,\bb C)$, $O(n,\bb C)$, on the variety of
nilpotent elements in the corresponding Lie algebra. All orbit
closures are normal in the case of $GL(n,\bb C)$ but not for the
other classical groups. We focus on the other classical groups from now on. 

Let $\CO$ be a nilpotent orbit, $\ovl{\CO}$ its closure,
and $R(\CO)$ and $R(\ovl{\CO})$ be the rings of regular functions. A
well known result states that
\begin{equation} \label{eq:normality}
R(\overline{\C O})\subseteq R(\C O),\ 
\text{and equality holds iff } \ovl{\C O}\ \text{is normal. }
\end{equation}
An important question (from our point of view at least) is what
functions on $R(\CO)$ extend to the closure, and if not, how far they
extend to orbits $\CO'\subset \ovl{\CO}$. In this paper and its
sequel, we provide an answer using the representation theory of
infinite dimensional representations of the
classical complex groups viewed as real Lie groups. The starting point
is the work of Brylinski \cite{Br} which associates to each nilpotent
orbit $\CO$ a $(\fk g,K)-$module (or more precisely a Dixmier algebra) $\C B(\ovl{\CO})$. This module
has a natural filtration such that the corresponding graded object is
$R(\ovl{\CO}).$  Being a $(\fk g_c, K_c)-$module (here the subscript $c$ denotes complexification), it decomposes with
finite multiplicity under $K_c\cong G$, where $G$ is viewed as a complex group. Under this identification,
$$
Gr(\C B(\ovl{\CO}))\cong R(\ovl{\CO}).
$$
The main tool is the $\Theta-$correspondence of dual reductive pairs
as defined by Howe, \cite{H}. Using the results in \cite{AB}, we
provide a complete composition series for $\C B(\ovl{\CO})$. As 
consequences, we obtain the following results, which in particular
answer the question posed earlier.
\begin{itemize}
\item There is an analogous Dixmier algebra $\C B(\CO)$ with associated
  graded object isomorphic to $R(\CO),$ endowed with an injective map
  $\C B(\ovl{\CO}) \hookrightarrow \C B(\CO).$
	We specify the
  composition factors of these modules explicitly.
\item We find explicit formulas for the $K_c-$types of the modules
  $\C B(\ovl{\CO})$ and $\C B(\CO).$
\item We obtain a new proof of the classification of normal closures
  of nilpotent orbits based on representation theory.
\item All the composition factors of $\C B(\ovl{\CO})$ and $\C B(\CO)$ are \textbf{unitary}. The modules $\C B(\CO)$ and
  $\C B(\ovl{\CO})$ fit into the general philosophy of the
  \textit{Orbit Method}, pioneered by Kostant, Kirillov and Souriau.    
\end{itemize}

\begin{comment}

Let $G = Sp(2n,\mathbb{C})$ or $O(n,\mathbb{C})$ be a complex
classical Lie group, with Lie algebra $\fk g.$ Classical nilpotent
varieties $\ovl{\C O}$, \ie Zariski closures of nilpotent orbits $\C
O\subset \mathfrak{g}$ are of interest in algebraic geometry and
representation theory. In this paper, we study the relation between
regular functions on closures of such orbits and representation
theory in detail. Consequences are new proofs of the classification of
normal nilpotent orbits, and relations to the `\emph{Orbit
  Method}'.  

\smallskip
In \cite{Br}, Ranee Brylinski associates to each closure of a classical nilpotent
orbit $\ovl{\CO}\subset \mathfrak{g}$  a Dixmier algebra 
$\mathcal{B}(\ovl{\C O})$.  In particular, $\mathcal{B}(\ovl{\C O})$ 
has a $(\fg_c, K_c)-$module structure. Its $K_c \cong G$ 
spectrum is isomorphic to $R(\overline{\C O})$, the ring of regular
functions of $\ovl{\C O}$ viewed as a module for $K_c\cong G.$ 
In this manuscript, we study the structure of $\mathcal{B}(\ovl{\C O})$ 
for all nilpotent orbits in the classical groups. 
\end{comment}

Note that the modules $\C B(\ovl{\CO})$ and $\C B(\CO)$ are generally neither
irreducible nor unitary. Nevertheless, they do have filtrations by unitary
modules. This is different from the other known versions of the orbit
method. We view this as a feature of the lack of normality of
$\ovl{\CO}$.

\medskip
In the case of quantizing $R(\CO),$ one considers the more general case
$R(\CO,\psi)$ with $\psi$ an equivariant  local system. A version using the
$\Theta-$correspondence is analyzed in \cite{B5}. It is not
completely clear how to  generalize a local system $\psi$ of $\CO$ to $\ovl{\CO}$. The
variety introduced by Kraft-Procesi, and used in this paper, provides
a less singular  cover of $\ovl{\CO}$, which we call $\ovl{\CO_{ns}}.$
For this cover an analogue of $\psi$ makes sense. As
a consequence of the results in this paper, one obtains a quantization
of $R(\ovl{\CO_{ns}},\psi)$. 

\bigskip
Here are more details on the results of this paper.
In \cite{B5}, the first author constructed a unipotent 
representation $\Pi(\C O)$ for each classical nilpotent orbit $\C O$,
and proved that it is a quantization model of the orbit.
We `deform' some characters in the inducing module of $\Pi(\C O)$, and obtain a 
family of $(\fg_c, K_c)-$modules $\Pi_t(\C O)$ for $t \geq 0$ with $\Pi_0(\C O) = \Pi(\C O)$ (Definition \ref{def:deform}). Deformation of this kind was studied by Vogan \cite[Section 6]{V2}
for the only non-normal nilpotent variety $\overline{\mathcal{O}_{(4220)}}$ in $\mathfrak{sp}(8,\mathbb{C})$.

%The Brylinski model $\C B(\ovl{\C O})$ is closely related to quantization
%of complex nilpotent orbits $\C O$, which has as its main feature  `\textit{attaching}' a unitarizable
%$(\fg_c,K_c)-$module to the orbits whose $K_c-$spectrum is equal to that of $R(\C O)$. 
%In Type $A$, a quantization model for all nilpotent orbits and their closures can be obtained by
%parabolic induction on the trivial representation of a suitable Levi subgroup of $G$. 

The first reducibility point of $\Pi_t(\C O)$ occurs at $t = 1$. Writing 
$\C B(\C O) := \Pi_{1}(\C O)$, the main theorem of the manuscript 
is the following:
\begin{theorem}[Theorem \ref{thm:main}]
Let $\C O$ be a classical nilpotnet orbit. There is an injective $(\fg_c,K_c)-$module morphism:
\begin{equation} \label{eq:CE}
\C E: \C B(\ovl{\C O}) \hookrightarrow \C B(\C O),
\end{equation}
whose image is the submodule of $\C B(\C O)$ generated by the
spherical vector. 
\end{theorem}
For a spherical module (a module containing the trivial $K-$type with
multiplicity one), we will call the submodule generated by the
spherical vector the ``\textit{cyclic submodule}''.
For $\ovl{\mathcal{O}_{(4220)}}$, the above theorem was conjectured in \cite[Section 6]{V2}. Therefore, our result verifies and generalizes Vogan's observation to all classical nilpotent orbits.

In view of \eqref{eq:normality}, one can determine the (non-)normality
of $\ovl{\C O}$ by checking whether \eqref{eq:CE} is a proper inclusion or not.
In Corollary \ref{cor:discrepancy}, we obtain the multiplicities of 
\emph{diminutive $K_c-$types} (Definition \ref{def:diminutive}) of $\C B(\ovl{\C O})$ and $\C B(\C O)$. This 
gives a new proof of the classification of normal nilpotent varieties
in the classical cases, 
and refines the theorem of \cite{KP2} by showing that the that whenever $\ovl{\C O}$ is not normal, the 
discrepancy between $R(\overline{\C O})$ and $R(\C O)$ occurs as early as diminutive $K_c-$types.

\smallskip
A list of candidates for the composition factors of $\C B(\ovl{\C O})$
and $\C B(\C O)$ are given in Proposition \ref{prop:parastructure}. The fact
that all the candidates have diminutive lowest $K_c-$types plays an
important role in the proof of Theorem 1.1.  
Since $\C B(\C O) = \Pi_{1}(\C O)$ occurs at the first reducibility point by
deforming the unitary representation  $\Pi(\C O)$, all its composition factors must be unitarizable. 
In a subsequent work, we will show which of these candidates appear in
the composition series of $\C B(\ovl{\C O})$ and $\C B(\C O)$. As
a result, the multiplicities of \emph{all} $K_c-$types in $R(\ovl{\C
  O})$ and $R(\C O)$ can be computed explicitly. Details will appear in a sequel to this paper.

\smallskip
On the other hand, the map $\C E$ in \eqref{eq:CE} can be seen as a `quantized' version of the inclusion \eqref{eq:normality}. It would be of interest to have a geometric interpretation of the quantized morphism $\C E$.
For example,
Losev  \cite{Lo} defined a `\emph{quantization parameter space}' $\mathfrak{P}$
for each nilpotent orbit $\C O$, so that for all $\mu \in \mathfrak{P}$, there is a filtered algebra $\mathcal{A}_{\mu}(\C O)$ equipped with a graded Poisson isomorphism satisfying $\mathrm{gr}(\mathcal{A}_{\mu}(\C O)) \cong R(\C O)$. 
Recently, \cite{LMM} proved that $\mathcal{A}_0(\C O) \cong \Pi(\C O)$ as $(\fg_c, K_c)-$modules for all classical nilpotent orbits.
More generally, there exists $\nu \in \mathfrak{P}$ such that $\Pi_{t}(\C O) \cong \mathcal{A}_{t\nu}(\C O)$ for all $t \geq 0$ (see Remark \ref{rmk:defquant}). It is therefore natural to ask about the relationship between the filtered Poisson algebra $\mathcal{A}_{\nu}(\C O)$ and the normality of $\ovl{\C O}$. We expect this will be helpful in studying certain non-normal exceptional nilpotent varieties with branching generic singularities classified in \cite{FJLS}.

%studied \emph{filtered quantizations} 
%of conical symplectic singularities, and applied it to the case of nilpotent orbits. 
%More explicitly, 

%%%%%%%%%%%%%%%%%%%
%%%SECTION 2%%%%%%%
%%%%%%%%%%%%%%%%%%%
\section{Preliminaries}
\label{sec:prelim}
\subsection{Langlands Parameters}

We recall the Langlands parametrization of irreducible
$(\fg_c, K_c)-$modules for a complex Lie group $G$ viewed as a real Lie
group. Fix a maximal compact subgroup $K$, and a pair $(B,H=TA)$
where $B$ is a real Borel subgroup and $H$ is a $\theta-$stable Cartan
subgroup such that $T=B\cap H$, and $A$ the complement stabilized by $\theta$.  

The {\it Langlands parameter} of any irreducible module is a pair $(\la_L;\la_R)$
 such that
  $\mu:=\la_L-\la_R$ is the parameter of a character of $T$ in the
  decomposition of the $\theta-$stable
  Cartan subalgebra $H=T\cdot A,$ and $\nu:=\la_L+\la_R$ the
  $A-$character. The \textit{principal series representation}
  associated to $(\la_L;\la_R)$ is the $(\mathfrak{g}_c, K_c)-$module
$$
X(\lambda_L,\lambda_R) =\Ind_B^G(e^{\mu}\otimes e^\nu \otimes
1)_{K-finite}.
$$
The symbol $\Ind$ refers to Harish-Chandra induction. 
The infinitesimal character, when $\fk g_c$ is identified with $\fk
g\times\fk g,$ is $(\la_L;\la_R).$ 
Since $e^{(\la_L;\la_R)}|_T = e^{\mu}$, 
$$
X(\lambda_L,\lambda_R)\mid_{K}=\Ind_T^K(e^{\mu}).
$$

Let $\ovl{X}(\la_L;\la_R)$ be the unique irreducible subquotient of
$X(\la_L;\la_R)$ containing  the $K_c-$type with extremal weight
$\mu=\la_L-\la_R.$ This is called the Langlands subquotient. 

\begin{proposition}[Parthasarathy-Rao-Varadarajan, Zhelobenko] \label{prop:BVprop}
Let $(\la_L;\la_R)$ and $(\la_L',\la_R') $ be parameters. The
following are equivalent:
\begin{itemize}
\item $(\la_L;\la_R)=(w\la_L';w\la_R')$ for some $w\in W.$ 
\item  $X(\la_L;\la_R)$ and $X(\lambda_L'; \lambda_R')$ have
  the same composition factors with same multiplicities. 
\item The Langlands subquotient of $X(\la_L;\la_R)$, written as 
${\ovl X}(\la_L;\la_R)$, is the same as that of $X (\lambda_L', \lambda_R')$.
\end{itemize}

Furthermore:
\begin{itemize}
\item  Every irreducible $(\mathfrak{g}_c, K_c)-$module is equivalent
  to some ${\ovl X}(\la_L;\la_R)$.
\item When $Re\ \nu$ is dominant (anti-dominant) with respect to the roots in $B,$
$X(\la_L;\la_R)$ has a unique irreducible quotient (submodule) identical to
$\ovl{X}(\la_L;\la_R)$; it is the image of the long intertwining operator
  given by integration (see \cite{Kn} for details).
\end{itemize}
\end{proposition}

The case of the orthogonal group is dealt with via Clifford theory. We use Weyl's parametrization of the
finite dimensional representations of the orthogonal groups (see page 6 of \cite{AB}). The highest weight of a $K_c-$type 
will be denoted 
\begin{equation} \label{eq:weylo}
\mu = (a_1 \geq \dots \geq a_n | \pm 1)\end{equation}
whenever $a_n = 0$, and the Langlands quotients will acquire a $\pm 1$ whenever the
corresponding lowest $K_c-$types are in different irreducible quotients.

\medskip
If we need to specify the group, the standard module and Langlands
quotient will acquire a subscript, \eg $X_G(\la_L;\la_R)$ or $\ovl{X}_G(\la_L;\la_R).$ 
Also, we only deal with {\it spherical} parameters 
$\la_L = \la_R = \la \in \mathfrak{h}^*$ for most parts of this paper. In such cases, we will write 
\begin{itemize}
\item $X_G(\la) := X_G(\la; \la) $\quad (or $X_G(\la; \la | +1)$ for orthogonal groups); and
\item $\ovl{X}_G(\la) := \ovl{X}_G(\la;\la)$\quad (or $\ovl{X}_G(\la; \la | +1)$ for orthogonal groups). 
\end{itemize}

\begin{definition} \label{def:diminutive}
Let $G$ be a classical complex Lie group, and $V_{\mu}$ be the irreducible, finite-dimensional $K_c-$type with highest weight $\mu$ (or Weyl's parametrization \eqref{eq:weylo} for orthogonal groups). The
{\bf diminutive $K_c-$types} of $G$ are:
$$\begin{cases}
V_{(1^{k},0^{n-2k},-1^{k})}\ \  ( 0 \leq k \leq \lfloor \frac{n}{2} \rfloor) & \text{for}\ G = GL(n,\mathbb{C}) \\
V_{(1^{k},0^{n-k}|\ (-1)^k)}\ \ ( 0 \leq k \leq n)  & \text{for}\ G = O(2n+1,\mathbb{C})\\
V_{(1^{2k},0^{n-2k})}\ \ ( 0 \leq k \leq \lfloor \frac{n}{2} \rfloor) & \text{for}\ G = Sp(2n,\mathbb{C}) \\
V_{(1^{2k},0^{n-2k}|\ \pm 1)}\ \ ( 0 \leq k \leq \lfloor \frac{n}{2} \rfloor) \  & \text{for}\ G = O(2n,\mathbb{C})
\end{cases}.$$ 
For instance, the diminutive $K_c-$types of $Sp(2n,\mathbb{C})$ are of the form $\wedge^{2\ell}\mathbb{C}^{2n}/\wedge^{2\ell-2}\mathbb{C}^{2n}$,
and the diminutive $K_c-$types of $O(2n+\delta,\mathbb{C})$ ($\delta \in \{0,1\}$) are of the form $\wedge^{2\ell} \mathbb{C}^{2n+\delta}$ for some
positive integer $\ell$.
% (If $\delta = 0$ and $k = n$, then there is only one diminutive $K_c-$type $V_{(1^n)}$). 

\smallskip
Let $\gamma: X \to Y$ be a $(\fg_c, K_c)-$morphism. We write 
$$\gamma: X \overset{dm}{\longrightarrow} Y$$ 
if $\gamma$ is an isomorphism when restricted to diminutive
  $K_c-$types.
\end{definition}

As in \cite{B4}, we will use the strings to denote the parameters $\la$:
\begin{definition}
Let $a, A \in \frac{1}{2}\mathbb{Z}$ be such that $A - a \in \mathbb{Z}$. A {\bf string} is an ascending sequence of numbers
$$(a \dots A) := (a,a+1,\dots,A-1,A)$$ 
(if $A - a < 0$ then the string is empty). 

The {\bf dual string} of $(a \dots A)$ is defined as
$$(a \dots A)^* := (-A \ldots -a)= (-A,-A+1,\dots,-a-1,-a),$$
so that the dual (or contragredient) representation of $\ovl{X}_{GL}((a \dots A))$ is $\ovl{X}_{GL}((a \dots A))^* = \ovl{X}_{GL}((a \dots A)^*)$.
\end{definition}

We also introduce some shorthand notation for induced modules:
\begin{definition} \label{def:ind}
 Let $G'$ be the Lie group with the same
type as $G$ of lower rank, and $\Psi$ be a representation of 
$G'$. We write
\[I^G\left(
  \la \boxtimes\cdots \boxtimes \la' \boxtimes\Psi\right) :=
\mathrm{Ind}_{GL \times \cdots \times GL \times G'}^{G}\left(\ovl{X}_{GL}(\la)
\boxtimes \cdots \boxtimes \ovl{X}_{GL}(
  \la') \boxtimes \Psi\right).
\]
Similarly, we write 
\[
I^G\left(
  \la \boxtimes\cdots \boxtimes\la' \right) :=
\mathrm{Ind}_{GL \times \cdots \times GL}^{G}\left(\ovl{X}_{GL}(\la) 
\boxtimes \cdots \boxtimes \ovl{X}_{GL}(\la')\right).
\]
\end{definition}

\subsection{An Intertwining Operator}
Let $(a \dots A)$, $(b \dots B)$ be two strings. In this manuscript, we make extensive use of the intertwining operator
\begin{equation} \label{eq:iota}
\iota: I^{GL}\left((a \dots A)\boxtimes(b \dots B)\right) \longrightarrow 
	I^{GL}\left((b \dots B)\boxtimes(a \dots A)\right),
\end{equation}
where $\iota$ is normalized so that it is identity on the trivial $K_c-$type. Following \cite[Section 2]{B4}, 
we say that the strings are \textbf{nested}, if one of the following conditions hold:
\begin{itemize}
\item[(i)] $a-b \notin \mathbb{Z}_{>0}$; or
\item[(ii)] $b \leq a \leq A \leq B$; or
\item[(iii)] $a-b \in \mathbb{Z}_{>0}$ and $B+1 < a$.
\end{itemize}
For example, the strings 
\begin{align*}
&& && && && (a,&& \dots,&& B&, && B+1, && B+2, && \dots, && A)\\
(b, && \dots,&& a-2, && a-1, &&a, && \dots, && B&)
\end{align*}
and
\begin{align*}
&(a = B+1,\ \dots,\ A)\\
(b,\ \dots,\ B&)
\end{align*}
are {\bf not} nested.
\begin{proposition} \label{p:glgl}
Consider the intertwining operator $\iota$ in \eqref{eq:iota}.
\begin{itemize}
\item[(a)] If the parameter is nested, then $\iota$ is an isomorphism
on the level of diminutive $K_c-$types, i.e. 
$$I^{GL}\left((a \dots A)\boxtimes(b \dots B)\right) \overset{dm}{\longrightarrow}
	I^{GL}\left((b \dots B)\boxtimes(a \dots A)\right).$$
	Moreover, if the strings satisfy (i) or (ii) above, then both modules are irreducible, and the map is an
	isomorphism on all $K_c-$types.
\item[(b)] If $a-b \in \mathbb{Z}_{>0}$ and $a \leq  B+1 \leq A$,
then the kernel of $\iota$ is the irreducible module with parameter
\[
\begin{pmatrix}
a-1&\dots &B  &a&\dots &B+1&B+2&\dots &A&b&\dots &a-2\\
a  &\dots &B+1&a-1&\dots &B&B+2&\dots &A&b&\dots &a-2 
\end{pmatrix}.
\]
In this case, the image of $\iota$ is equal to
\[
I^{GL}\left((b \dots A)\boxtimes(a \dots B)\right)
\]
\end{itemize}
\end{proposition}
\begin{proof} Detailed calculations for intertwining operators an diminutive $K_c-$types can be found in \cite{B4}; they exploit the relations between diminutive $K_c-$types and Weyl group representations. The composition factors of modules of $GL(m+n)$ induced from characters on $GL(m)\times GL(n)$ can be obtained from the techniques employed in \cite{BSS}. We omit further details.
\end{proof}

\subsection{Nilpotent Orbits}\label{sec:nilpotent}
A classical nilpotent orbit $\CO$ is denoted by its 
Young diagram. We specify $\C O$ by the
{\it columns} of its Young diagram. 
This parametrization is already employed in Kraft and Procesi \cite{KP2}. 
Our results are best phrased in these terms as well.

\smallskip
We describe the nilpotent orbits for symplectic groups $G = Sp(2n,\mathbb{C})$ first.
Denote the columns of the Young diagram by 
\begin{equation} \label{eq:corbits}
(c_0 \geq c_1 \geq \dots \geq c_{2p} \geq c_{2p+1})\quad (\text{set}\ c_{2p+1}=0\ \text{if necessary})
\end{equation}
with $\sum_i c_i = 2n$. The condition that $\CO$ is a symplectic orbit translates into
$c_{2i} + c_{2i+1}$ is even. A {\it special}
nilpotent orbit in the sense of \cite{Lu} satisfies the additional
condition that if $c_{2i}$ is odd, then $c_{2i} = c_{2i+1}$. 

\smallskip
There are similar characterizations of nilpotent orbits for orthogonal groups $G = O(2n+\delta,\mathbb{C})$ ($\delta = 0$ or $1$).
Namely, suppose the columns of the Young diagram 
are given by 
\begin{equation} \label{eq:bdorbits}
(c_1 \geq c_2 \geq \dots \geq c_{2p} \geq c_{2p+1})\quad (\text{set}\ c_{2p+1}=0\ \text{if necessary})
\end{equation}
with $\sum_i c_i = 2n + \delta$. Then it corresponds to an orthogonal nilpotent orbit
iff $c_{2i} + c_{2i+1}$ is even for $1 \leq i \leq p$. Moreover, the special
nilpotent orbits satisfies the additional
condition that if $c_{2i} \equiv \delta + 1\ (\mathrm{mod}\ 2)$, then $c_{2i} = c_{2i+1}$.

\medskip
From now on, we will specify any classical nilpotent orbit $\C O$ by the column sizes of its corresponding
Young diagram given by \eqref{eq:corbits} or \eqref{eq:bdorbits}.

\begin{definition}
We call $\C O$ {\bf odd} or {\bf even} if its column sizes $c_i$ are all odd or all even respectively.
\end{definition}

We recall the necessary and sufficient conditions for a classical nilpotent variety 
orbit $\overline{\C O}$ to be non-normal:
\begin{theorem}[\cite{KP2}] \label{thm:KP}
A classical nilpotent orbit $\C O$ of the form \eqref{eq:corbits} or \eqref{eq:bdorbits} has non-normal
closure if and only if there exists $1 \leq i \leq j \leq k$ such that
\[
c_{2i-2} > c_{2i-1} =
c_{2i} = \dots = c_{2j-1} =
c_{2j}  > c_{2j+1}.
\]
(omit the condition on $c_{2i-2}$ in orthogonal groups when $i=1$). For instance, the first example for a symplectic orbit having non-normal closure is $\C O = (4 > 2 =2 > 0)$
in $\mathfrak{sp}(8,\mathbb{C})$ as mentioned in the introduction.
\end{theorem}

\begin{proof} This is a rephrasing of the main result in \cite{KP2}.
Let $\C O$ be of the form
$$
\C O = (\cdots c_{2i-2} > c_{2i-1} =c_{2i} = \dots = c_{2j-1}=c_{2j} > c_{2j+1} \cdots).
$$
Then the orbit $\C O'$ with 
$$
\C O^{\flat} = (\cdots c_{2i-2} \geq c_{2i-1}+2 > c_{2i} = \dots = c_{2j-1} > c_{2j}-2 \geq c_{2j+1} \cdots)
$$
occurs in the closure of $\C O.$ Table I of
\cite{KP2} states that $\C O^{\flat}$ is a minimal $\epsilon-$degeneration of type
$A_{2n-1} \cup A_{2n-1}$ with $n = 2(i-j+1)$.  By Theorem 2 of \cite{KP2}, this is the only
type of non-normal singularity that can occur in $\ovl{\C O}$. 
\end{proof}

\begin{example}
The closure of $\C O = (8,8,6,6,6,4,4,2)$ in $\mathfrak{sp}(44,\mathbb{C})$ is not normal, since
$
c_4=6>c_5=4=c_6=4>c_7=2.
$  

In fact, under the setting of Definition 3.3(a) of \cite{KP2} and
the main results therein, the singularity type of the orbit 
$$\C O^{\flat} = (8,8,6,6,6,6,2,2) \subset \ovl{\C O} = \ovl{(8,8,6,6,6,4,4,2)}$$
is equivalent to that of
$$(6,2,2) \subset \ovl{(4,4,2)}$$ 
in $\mathfrak{o}(10,\bb{C})$
by removing the common columns of $\C O^{\flat}$ and $\C O$ on the left,
which in turn is equivalent to that of
$$(4) \subset \ovl{(2,2)}$$ 
in $\mathfrak{o}(4,\mathbb{C})$
by removing the first two common rows. This
becomes a minimal irreducible $\epsilon-$degeneration by
Definition 3.3(b) of \cite{KP2}, whose singularity is of type $A_1 \cup A_1$
from Table I of \cite{KP2}, which is not normal.
\end{example}

\subsection{Unipotent Representations} \label{sec:unip}
In this section, we recall the content of \cite{B5} for a quantization model $\Pi(\C O)$
of all classical orbits $\C O$.
\begin{definition} \label{def:roprime}
Let $\C O$ be a classical nilpotent orbit of the form
\begin{equation} \label{eq:o}
\C O = \begin{cases} 
(c_0 \geq c_1 \geq \dots \geq c_{2p+1}) & \text{if}\ G = Sp(2n,\mathbb{C}), \\
(c_1 \geq c_2 \geq \dots \geq c_{2p+1}) & \text{if}\ G = O(2n+\delta,\mathbb{C})
\end{cases},
\end{equation}
and let $\tau(\C O) := \{i \ |\ c_{2i-1} = c_{2i}\}$. Define $\C O'$ by
\begin{equation} \label{eq:oprime}
\C O' := \begin{cases} (c_0' \geq c_1' \geq \dots \geq c_{2q+1}') & \text{if}\ G = Sp(2n,\mathbb{C}), \\
(c_1' \geq c_2' \geq \dots \geq c_{2q+1}')& \text{if}\ G = O(2n+\delta,\mathbb{C})
\end{cases},
\end{equation}
where $\C O'$ is obtained from $\C O$ by removing all $c_{2i-1}, c_{2i}$ for $i \in \tau(\C O)$ 
in $\C O$. 
Then $\C O'$ is a nilpotent orbit in $\mathfrak{g}'$ of the same type as $\fg$ but of smaller rank. 
\end{definition}

\begin{example}
Consider $\C O = (8,8,6,{\bf 6,6,4,4},2)$ in $\mathfrak{sp}(44,\mathbb{C})$ as in the previous example. Then $\tau(\C O) = \{2,3\}$
and $\C O' = (8,8,6,2)$.
\end{example}

Using the notations in Definition \ref{def:roprime}, the {\bf spherical unipotent representation} attached to a classical nilpotent orbit $\C O$ is
\begin{equation} \label{eq:sunip}
\Pi(\C O) := \mathrm{Ind}_{\prod_{i \in \tau(\C O)} GL(c_{2i}) \times G'}^{G}\left(\mathrm{triv} \boxtimes \dots \boxtimes \mathrm{triv} \boxtimes
\mathcal{U}(\mathcal{O}')\right),
\end{equation}
where
$$\mathcal{U}(\mathcal{O}') := \begin{cases}
\overline{X}_{G'}\left(\la[c_0',c_1'], \dots,  \la[c_{2q}',c_{2q+1}']\right) &\text{if}\ G' = Sp(2n',\mathbb{C})\\
\overline{X}_{G'}\left(\la[-\delta,c_1'], \la[c_2',c_3'], \dots,  \la[c_{2q}',c_{2q+1}']\right) &\text{if}\ G' = O(2n'+\delta,\mathbb{C})\end{cases}$$
and $\la[x,y]$ is given by
\begin{equation} \label{eq:laxy}
\la[x,y] := \left(-\frac{y}{2} +1 \ldots \frac{x}{2}\right)
\end{equation}
for $x, y \in \mathbb{Z}_{\geq -1}$.

In \cite{B1} and \cite{B5}, it was shown that
$\Pi(\C O)$ is an irreducible, unitary $(\fg, K)-$module, whose $K_c \cong G-$spectrum satisfies $\Pi(\C O)|_{K_c} \cong R(\C O)$.
In order to relate $\Pi(\C O)$ with the Brylinski model $\C B(\ovl{\C O})$, one needs to `deform' $\Pi(\C O)$ as follows:
\begin{definition} \label{def:deform}
Retain the notations above. For $t \geq 0$, define
$$\Pi_{t}(\C O) := \mathrm{Ind}_{\prod_{i \in \tau(\C O)} GL(c_{2i}) \times G'}^{G}\left(|\det|^{t} \boxtimes \dots \boxtimes |\det|^{t} \boxtimes \mathcal{U}(\mathcal{O}')\right),$$
and $$\C B(\C O) := \Pi_{1}(\C O)$$
(if $\tau(\C O)$ is empty set and $\C O = \C O'$, then $\C B(\C O) = \C U(\C O)$).
\end{definition}
Obviously, when $t = 0$, $\Pi_{0}(\C O) = \Pi(\C O)$. Also note that $\Pi_{t}(\C O)$ has the same $K_c-$spectrum
for all $t \geq 0$. 
 
\begin{remark} \label{rmk:defquant}
We now explain the relation between $\Pi_{t}(\C O)$ and deformation quantization. For all nilpotent orbits $\C O$, Losev in \cite{Lo} 
constructed a quantization parameter space $\mathfrak{P}$ such that for each $\mu \in \mathfrak{P}$,
there exists a filtered quantization $\mathcal{A}_{\mu}(\C O)$ satisfying $\mathrm{gr}(\mathcal{A}_{\mu}(\C O)) \cong R(\C O)$ 
(see \cite[Section 4]{LMM} for more details). In the special case when $\C O$ is a classical nilpotent orbit, it is proved in Section 10 of \cite{LMM} 
that there exists a $\nu \in \mathfrak{P}$ such that for all $t \geq 0$, we have a $(\fg_c, K_c)-$module isomorphism
$\Pi_{t}(\C O) \cong \mathcal{A}_{t\nu}(\C O).$
\end{remark}

\section{Dual Pairs and the Kraft-Procesi Model}
\label{sec:indprinciple}
\subsection{The Kraft-Procesi Model $\mathcal{B}(\ovl{\C O})$}
The Kraft-Procesi model, introduced in Definition 6.1
in \cite{Br}, is an admissible
$(\fk g_c, K_c)-$module whose $K_c-$structure matches $R(\ovl{\C O}).$

Given a classical nilpotent orbit $\C O$ as in \eqref{eq:o}, let $G_{2i} = Sp(c_{2i}+\dots +c_{2p+1},\bC)$ and $G_{2i+1}=O(c_{2i+1}+\dots
+c_{2p+1},\bC).$  Let $\fk g_j$ be the corresponding Lie algebras, and
$ K_j$ the maximal compact subgroups. Then 
\begin{itemize}
\item $G = G_{\alpha}$, where $\alpha = 0$ if $G$ is symplectic and $\alpha =1$ if $G$ is orthogonal.
\item $G_j\times G_{j+1}$ is a reductive dual pair, and let $\Omega_{j+1}$ be the oscillator representation as in \cite{AB}. 
\end{itemize}

\begin{definition}[Kraft-Procesi Model]\label{def:kp}
Let $\C O$ be a classical nilpotent orbit given in \eqref{eq:o}. Using the notations above, define a $((\fk g_{\alpha})_c,(K_{\alpha})_c) = (\fk g_c,K_c)-$module $\Omega(\ovl{\C O})$ by:
\[
\Omega(\ovl{\C O}) := \Omega_{\alpha+1} \otimes \Omega_{\alpha+2} \otimes \dots \otimes \Omega_{2p+1} / (\fg_{\alpha+1} \times \dots \times \fg_{2p+1})(\Omega_{\alpha+1} \otimes \dots \otimes \Omega_{2p+1}).
\]
\end{definition}
Let $K^1:=K_{\alpha+1} \times \ldots \times K_{2p},$ and $(K^1)^0$ be the connected component of the identity. 
Then $(K^1)^0$ acts trivially on $\Omega(\cCO).$ So $K^1/(K^1)^0$
acts on $\Omega(\cCO).$ 
\begin{theorem}[\cite{Br}, Theorem 6.3]\label{t:br}
Let $\C B(\cCO) := \Omega(\cCO)^{K^1/(K^1)^0}$, then
$$
\C B(\cCO)\mid_{K_c} \cong R(\cCO)\mid_{G},
$$ 
where on the right hand side the complex group $G$ is
  identified with $K_c.$
\end{theorem}

\subsection{The Induction Principle for Dual Pairs}
Since the definition of $\mathcal{B}(\ovl{\C O})$ involves oscillator representations, we first recall some results in \cite{AB} on the dual pair correspondence:
\begin{theorem}[\cite{AB}, Corollary 3.21] \label{thm:abinduction}
Let $H_1\times H_2$ be a reductive dual pair, and
$P=MN=P_1\times P_2=M_1N_1\times M_2N_2\subset H_1\times H_2$ a real
parabolic subgroup:
\begin{equation*}
  \label{eq:dpairosp}
  \begin{aligned}
    &H_1=O(2m+\tau), &&M_1=GL(k)\times O(2m+\tau-2k);\\
    &H_2=Sp(2n),     &&M_2=GL(\ell)\times Sp(2n-2\ell);
  \end{aligned}
\end{equation*}
with $\tau =0,1.$ Suppose there is a non-trivial $M-$equivariant map
$\Omega_M\longrightarrow\sig_1\boxtimes\sig_2$. Then there is a
non-trivial map
\begin{equation*}
  \label{eq:indpr}
  \begin{aligned}
  \Omega\longrightarrow \Ind_{P_1}^{H_1}\left(\al_1\sig_1\right)&\boxtimes\Ind_{P_2}^{H_2}\left(\al_2\sig_2\right), \text{ where}\\
  \al_1(g_1)=|\det g_1 |^{2n-2m-\tau+k-\ell +1}&, \ \ \ \al_2(g_2)=|\det g_2 |^{-2n+2m+\tau-k+\ell -1}.
  \end{aligned}
\end{equation*}
\end{theorem}

\begin{remark}
The exponents in $\al_1,\ \al_2$ are negatives of the ones in \cite{AB}, because the parabolic
subgroups in \cite{AB} are opposite from the ones used here.
\end{remark}

The following corollary is an almost direct consequence of Theorem \ref{thm:abinduction}:
\begin{corollary}  \label{cor:abstring}  \mbox{}\\
\noindent (a)\ Let $G_i \times G_{j}$ be a dual pair, with $G_i'$ and $G_{j}'$ 
be of the same type of $G_i$ and $G_{j}$ respectively such that
\begin{align*}
L_i &= GL(\alpha) \times \dots \times GL(\beta) \times G_i'\\
L_{j} &= GL(\alpha) \times \dots \times GL(\beta) \times G_{j}'
\end{align*}
are Levi subgroups of $G_i$ and $G_{j}$ respectively.
Suppose there is a non-trivial $G_i' \times G_j'-$equivariant map
$$\Omega' \longrightarrow \Pi_i' \boxtimes \Pi_{j}'$$
corresponding to the dual pair $G_i' \times G_{j}'$, with
$\Pi_i',\Pi_{j}'$ spherical, and the spherical vector in the
image. Then there is a non-trivial intertwining operator
\begin{align*}
\Omega \longrightarrow I^{G_i}\left((
    a \dots A) \boxtimes\cdots \boxtimes(
    b \dots B)\boxtimes\ \Pi_i' \right) \boxtimes I^{G_{j}}\left((
    a \dots A)^* \boxtimes\cdots  \boxtimes(
    b \dots B)^* \boxtimes\ \Pi_{j}' \right),
\end{align*}
to the induced module with $\alpha = A-a+1$, $\dots$, $\beta = B-b+1$. Moreover, the spherical vector is in its image.

\medskip
\noindent (b)\ Let $k >0$ be an integer such that $k \equiv \delta\ (\mathrm{mod}\ 2)$
and $O(2r+\delta) \times Sp(2r+\delta+k)$ forms a dual pair.
Suppose there is a non-trivial $O(2r+\delta) \times Sp(2r+\delta+k)-$equivariant map
$$
\Omega' \longrightarrow \Pi \boxtimes \Pi'
$$
with the spherical vector in its image.
Then there is an intertwining operator corresponding to the dual pair
$O(2r+\delta+2k) \times Sp(2r+\delta+k)$
$$
\Omega \longrightarrow I^{O(2r+\delta+2k)}\left(\left(-\frac{k}{2}+1 \ldots
  \frac{k}{2}\right)\boxtimes\ \Pi \right) \boxtimes \Pi'
$$
with the spherical vector in its image.
\end{corollary}

%{\clrr 
  %I do think it needs to say intertwining operator everywhere; it
  %means the map commutes  with the $(\fk g_c,K)-$action. You wrote
  %$M-$equivariant at the top of the page; it should  say $(\fk
  %m_cmK\cap M)-$equivariant, and you can use that instead of
  %intertwining operator. It can't be just a non-trivial map. Perhaps one place for the whole section where it says maps are intertwining maps.} {\clrblu corrected}
\begin{proof}
(a) follows  by induction on the number of $GL-$strings. The initial
case is the fact that the oscillator representation for the 
dual pair $GL(m) \times GL(m)$ has a quotient
$$\Omega_{GL} \longrightarrow \chi \boxtimes \chi^*$$
for any 1-dimensional character $\chi$ of $GL(m)$ (cf. \cite[Proposition 2.2(1)]{AB}).

\medskip
(b) Consider 
  \begin{align*}
    &H_1=O(2r+\delta + 2k), && M_1=GL(k) \times O(2r+\delta);\\
    &H_2=Sp(2r+\delta + k), && M_2=GL(0) \times Sp(2r+\delta+k);
  \end{align*}
and $\Omega_M \rightarrow (\mathrm{triv}_{GL(k)} \otimes \Pi) \boxtimes \Pi'$. Then 
\[
2n-2m -\tau + k - \ell + 1 = (2r+\delta+k) - (2r + \delta + 2k) + k - 0 + 1 = 1
\]
and Theorem \ref{thm:abinduction} implies there is a non-trivial map
\begin{align*}
\Omega  \to 
I^{H_1}\left(
   \left(- \frac{k-1}{2} + \frac{1}{2} \dots \frac{k-1}{2}+ \frac{1}{2}\right)\boxtimes\ \Pi \right) \boxtimes  \Pi'
	= I^{H_1}\left(\left(-\frac{k}{2}+1 \ldots \frac{k}{2}\right)\boxtimes\ \Pi\right) \boxtimes \Pi'.
\end{align*}
\end{proof}

\begin{proposition} \label{prop:xbaromap}
Let $\C O$ be a classical nilpotent orbit. Then there is a $(\fg_c, K_c)-$module map
\[
\Omega(\ovl{\C O}) \longrightarrow \C B(\C O).
\]
with the spherical vector in its image.
\end{proposition}

\begin{proof} The proof is  by induction on the number of columns of $\C O$ as
  defined  in \eqref{eq:o}. When $p = 0$, consider the orbits $\C Q_0 = (c_0, c_1)$ in $Sp(c_0+c_1,\mathbb{C})$, and the zero orbit $\C Q_1 = (c_1)$ in $O(c_1, \mathbb{C})$. It is obvious that
$$\Omega(\ovl{\C Q_1}) = \mathcal{U}(\C Q_1) = \C B(\C Q_1)$$
are all equal to the trivial representation of $O(c_1,\mathbb{C})$, so the proposition holds for $\C Q_1$. 

For $\C Q_0$, consider the pair $O(c_1)\times Sp(c_0 + c_1)$. By \cite[Theorem 3.5.1]{B5}, the unipotent representations
$\mathcal{U}(\C Q_1) \leftrightarrow \mathcal{U}(\C Q_0)$ are in dual pair correspondence. So
there is a non-trivial $O(c_1) \times Sp(c_0+c_1)-$equivariant map
\begin{align*}
\Omega_1 \longrightarrow \mathcal{U}(\C Q_1) \boxtimes \mathcal{U}(\C Q_0).
\end{align*}
Since $\mathcal{U}(\C Q_1)$ is the trivial representation, the above
map factors through $\mathfrak{g}_1\Omega_1$, and there is a
non-trivial $(\fk g_c, K_c)-$equivariant map
$$
\Omega(\ovl{\C Q_0}) = \Omega_1/\mathfrak{g}_1\Omega_1 \longrightarrow \mathcal{U}(\C Q_0) = \C B(\C Q_0) .
$$
as claimed in the proposition. 

\smallskip
%We do the induction step. {\clrr I assume here that there are groups $G_0, G_1, G_2$ and one pair is $Sp\times O$ and the other $O\times Sp?$ If so, a few words of explanation would help. If not, I am a bit lost.} {\clrblu I will stick with $O \times Sp$ in the whole paper.}

Suppose the proposition holds for $\C O_2 := (c_2\geq c_3 \geq \dots
\geq c_{2p} \geq c_{2p+1})$ in $G_2$. We will prove that it also holds for $\C O_1 := (c_1\geq c_2 \geq \dots \geq c_{2p} \geq c_{2p+1})$ in $G_1$ and $\C O_0 := (c_0\geq c_1 \geq \dots \geq c_{2p} \geq c_{2p+1})$ in $G_0$.

Let 
$$\tau(\C O_2) := \{i_r, \dots, i_1\}\quad \text{and}
\quad \C O_2' := (c_2' \geq c_3' \geq \dots \geq c_{2q+1}')$$ 
be as defined in \eqref{eq:oprime}.
By induction hypothesis, there is a non-trivial $((\fk g_2)_c, (K_2)_c)-$equivariant map
\begin{align*}
\Omega(\ovl{\C O_2}) := \Omega_3 \otimes\dots \otimes \Omega_{2p+1}/ (\fg_3 \times &\dots \times \fg_{2p+1})(\Omega_3 \otimes \dots \otimes \Omega_{2p+1})  \longrightarrow \\ 
&\C B(\C O_2) := I^{G_2}\left(\la[c_{2i_1},c_{2i_1}] \boxtimes \cdots \boxtimes \la[c_{2i_r},c_{2i_r}] \boxtimes \mathcal{U}(\C O_2') \right),
\end{align*}
noting that $\ovl{X}_{GL}(\la[a,a]) = \ovl{X}_{GL(a)}\left(-\frac{a}{2}+1 \ldots \frac{a}{2}\right) = |\det|^1$ by \eqref{eq:laxy}.

\smallskip
To prove the proposition for $\C O_1 := (c_1\geq c_2 \geq \dots \geq c_{2p} \geq c_{2p+1})$ in $G_1$, one needs to consider the following two cases:

\smallskip
\noindent \underline{\bf Case I: $c_1 \neq c_2 (= c_2')$.} In this case, $\tau(\C O_1) = \tau(\C O_2)$ and
$\C O_1' := (c_1 \geq c_2' \geq c_3' \geq \dots \geq c_{2q+1}')$. 
By \cite[Theorem 3.5.1]{B5}, there is a non-trivial $G_1' \times G_2'-$equivariant map
$$\Omega' \longrightarrow \mathcal{U}(\C O_1') \boxtimes \mathcal{U}(\C O_2')$$
to the appropriate unipotent representations. Applying Corollary \ref{cor:abstring}(a) to the above correspondence, there is a non-trivial $G_1 \times G_2-$equivariant map
$$
\Omega_2 \longrightarrow \begin{matrix} I^{G_1}\left(\la[c_{2i_1},c_{2i_1}] \boxtimes \cdots \boxtimes \la[c_{2i_r},c_{2i_r}] \boxtimes \mathcal{U}(\C O_1') \right)\\ \boxtimes \\  I^{G_2}\left(\la[c_{2i_1},c_{2i_1}]^* \boxtimes \cdots \boxtimes \la[c_{2i_r},c_{2i_r}]^*\boxtimes \mathcal{U}(\C O_2') \right) \end{matrix} =\ \C B(\C O_1) \boxtimes \C B(\C O_2)^*.
$$
Therefore, we have the non-trivial $G_1 \times G_2-$equivariant map
\begin{align*}
\Omega_2 \otimes (\Omega_3 \otimes \dots \otimes \Omega_{2p+1}/ (\fg_3 \times \dots \times \fg_{2p+1})(\Omega_3 \otimes \dots \otimes \Omega_{2p+1})) &\longrightarrow (\C B(\C O_1) \boxtimes \C B(\C O_2)^*) \otimes \C B(\C O_2) \\
&\longrightarrow \C B(\C O_1) \boxtimes \mathrm{triv}_{G_2},
\end{align*}
where the last $\longrightarrow$ is the contraction map $\C B(\C O_2)^* \otimes \C B(\C O_2) \to \mathrm{triv}_{G_2}$. As a consequence, the map factors through $\fg_2\Omega_2$ and one has
\begin{align*}
\Omega(\ovl{\C O_1}) := \Omega_2 \otimes \dots \otimes \Omega_{2p+1} / (\fg_2 \times \fg_3 \times \dots \times \fg_{2p+1})(\Omega_2 \otimes \Omega_3 \otimes \dots \otimes \Omega_{2p+1}) \longrightarrow \C B(\C O_1)
\end{align*}
as in the proposition.

\medskip
\noindent \underline{\bf Case II: $c_1 = c_2 (= c_2')$.} Then $\tau(\C O_1) = \tau(\C O_2) \cup \{1\}$, and
$\C O_1' := (c_3' \geq \dots \geq c_{2q+1}')$. 
By \cite[Theorem 3.5.1]{B5} again, there is a non-trivial $G_3' \times G_2'-$equivariant map
$$\Omega' \longrightarrow \mathcal{U}(\C O_1') \boxtimes \mathcal{U}(\C O_2').$$
Now apply Corollary \ref{cor:abstring}(b) with $k = c_2$ to the above correspondence to get a non-trivial there is a non-trivial $G_1' \times G_2'-$equivariant map
$$
\Omega'' \longrightarrow I^{O}\left(\la[c_{2},c_{2}] \boxtimes \mathcal{U}(\C O_1') \right)\\ \boxtimes \mathcal{U}(\C O_2').
$$
Finally, using Corollary \ref{cor:abstring}(a) on the above correspondence, one has a non-trivial $G_1 \times G_2-$equivariant map
\begin{align*}
\Omega_2 \longrightarrow\ &\begin{matrix} I^{G_1}\left(\la[c_{2i_2},c_{2i_2}] \boxtimes \cdots \boxtimes \la[c_{2i_r},c_{2i_r}]\ \boxtimes \la[c_{2},c_{2}] \boxtimes \mathcal{U}(\C O_1') \right)\\ \boxtimes \\  I^{G_2}\left(\la[c_{2i_2},c_{2i_2}]^* \boxtimes \cdots \boxtimes \la[c_{2i_r},c_{2i_r}]^* |\ \mathcal{U}(\C O_2') \right) \end{matrix}\\ 
\xrightarrow{\ dm}\ &\begin{matrix} I^{G_1}\left( \la[c_{2},c_{2}] \boxtimes \la[c_{2i_2},c_{2i_2}] \boxtimes \cdots \boxtimes \la[c_{2i_r},c_{2i_r}]\ \boxtimes \mathcal{U}(\C O_1') \right)\\ \boxtimes \\  I^{G_2}\left(\la[c_{2i_2},c_{2i_2}]^* \boxtimes \cdots \boxtimes \la[c_{2i_r},c_{2i_r}]^* |\ \mathcal{U}(\C O_2') \right) \end{matrix} =\ \C B(\C O_1) \boxtimes \C B(\C O_2)^*,
\end{align*}
where the $\stackrel{dm}{\longrightarrow}$ comes from applying the intertwining operator $\iota$ between the nested strings $\la[c_{2},c_{2}]$ and $\la[c_{2i_j},c_{2i_j}]$ on the first induced module.
Consequently, the proposition follows from the arguments in Case I. 
%{\clrr I am
  %not sure what the notation is in the above formula; one induced
  %module below another.} {\clrblu $ \begin{matrix} A\\ \boxtimes \\ B \end{matrix} := A \boxtimes B$. I ran out of space to write it in one line. It is an easy fix...}

\medskip
As for $\C O = (c_0 \geq c_1 \geq \dots \geq c_{2p} \geq c_{2p+1})$ in $G_0$, note that
$$\Omega' \longrightarrow \mathcal{U}(\C O_1') \boxtimes \mathcal{U}(\C O_0')$$ 
are in dual pair correspondence by \cite[Theorem 3.5.1]{B5}, so one can apply Corollary \ref{cor:abstring}(a) as before to conclude that
$$\Omega_1 \longrightarrow \C B(\C O_1)^* \boxtimes \C B(\C O_0).$$
Then the contraction argument above applies, so that one has
$\Omega(\ovl{\C O_0}) \longrightarrow \C B(\C O_0)$ as stated in the proposition.
\end{proof}

\begin{example}
We present an example of the above proposition for $\C O = (8,{\bf 6,6},4,2,0)$ in $\mathfrak{sp}(26,\mathbb{C})$. Note that $c_1 = 6 = c_2 = 6$ are in bold, so that $\C O' = (8,4,2,0)$.

The intermediate orbits and dual pairs are given by:

\begin{align*}
&G_4 = Sp(2,\mathbb{C}); && G_3 = O(6,\mathbb{C}); && G_2 = Sp(12,\mathbb{C});
&& G_1 = O(18,\mathbb{C}); && G_0 = Sp(26,\mathbb{C}); \\
&\C O_4 = (20); && \C O_3 = (420); && \C O_2 = (6420);
&& \C O_1 = ({\bf 66}420); && \C O_0 = (8{\bf 66}420)\\
&\C O_4' = (20); && \C O_3' = (420); && \C O_2' = (6420);
&& \C O_1' = (420); && \C O_0' = (8420).
\end{align*}
Begin with the trivial representation
$$\Omega(\ovl{\C O_4}) = \mathcal{U}(\C O_4') = \ovl{X}_{G_4}(1) = \C B(\C O_4).$$
By \cite[Theorem 3.5.1]{B5}, the dual pair correspondence $G_3 \times G_4$ gives
$$\Omega_4 \longrightarrow \mathcal{U}(\C O_3') \boxtimes\mathcal{U}(\C O_4') = \ovl{X}_{G_3}(-1,0,1) \boxtimes \ovl{X}_{G_4}(1)$$
Since the last module of  $G_4$ is trivial, the map factors through $\mathfrak{g}_4\Omega_4$ and hence we have
\begin{equation} \label{eq:o3}
\Omega(\ovl{\C O_3}) := \Omega_4/\mathfrak{g}_4\Omega_4\longrightarrow \mathcal{U}(\C O_3') = \C B(\C O_3).
\end{equation}

Continue with the dual pair $G_3 \times G_2$. \cite[Theorem 3.5.1]{B5} gives
$$\Omega_3 \longrightarrow \mathcal{U}(\C O_3') \boxtimes  \mathcal{U}(\C O_2') = \ovl{X}_{G_3}(-1,0,1) \boxtimes \ovl{X}_{G_2}(-1,0,1;1,2,3).$$
Note that $\mathcal{U}(\C O_3')^* = \mathcal{U}(\C O_3')$ is self-dual. Combining \eqref{eq:o3} with the above map, one has
$$\Omega_3 \otimes \Omega_4/\mathfrak{g}_4\Omega_4 \longrightarrow \left( \mathcal{U}(\C O_3') \otimes\mathcal{U}(\C O_3')^* \right) \boxtimes \mathcal{U}(\C O_2') \longrightarrow \mathrm{triv}_{G_3} \boxtimes \mathcal{U}(\C O_2')$$
and hence it factors through $\mathfrak{g}_3\Omega_3$ so that
\begin{equation} \label{eq:o2}
\Omega(\ovl{\C O_2}) := \Omega_3 \otimes \Omega_4/(\mathfrak{g}_3 \times \mathfrak{g}_4)\Omega_3 \otimes \Omega_4 \longrightarrow \mathcal{U}(\C O_2') = \C B(\C O_2).
\end{equation}

The case for $\C O_1$ is different, since $\C O_1 \neq \C O_1'$ and we are in \underline{\bf Case II} of the above proof. Instead of considering $G_1 \times G_2$ directly, first consider the dual pair $G_1' \times G_2 = G_3 \times G_2 = O(6,\mathbb{C}) \times Sp(12,\mathbb{C})$
where we have a non-trivial map
$$\Omega' \longrightarrow \mathcal{U}(\C O_1') \boxtimes \mathcal{U}(\C O_2')$$
as above. By Corollary \ref{cor:abstring}(b), the dual pair $G_1 \times G_2$ has a non-trivial map
$$\Omega_2 \longrightarrow I^{G_1}\left((-2\ldots3) \boxtimes \ \mathcal{U}(\C O_1')\right) \boxtimes \mathcal{U}(\C O_2') = I^{G_1}\left(\la[6,6] \boxtimes \mathcal{U}(\C O_1')\right) \boxtimes \mathcal{U}(\C O_2')$$
Using \eqref{eq:o2} and $\mathcal{U}(\C O_2')^* = \mathcal{U}(\C O_2')$, one has
\begin{align*}
\Omega_2 \otimes (\Omega_3 \otimes \Omega_4/(\mathfrak{g}_3 \times \mathfrak{g}_4)\Omega_3 \otimes \Omega_4) &\longrightarrow I^{G_1}\left(\la[6,6] \boxtimes \mathcal{U}(\C O_1')\right) \boxtimes \mathcal{U}(\C O_2') \otimes \mathcal{U}(\C O_2')^* \\
&\longrightarrow I^{G_1}\left(\la[6,6] \boxtimes \mathcal{U}(\C O_1')\right) \boxtimes \mathrm{triv}_{G_2}
\end{align*}
and hence the map factors through $\mathfrak{g}_2\Omega_2$:
\begin{equation} \label{eq:o1}
\Omega(\ovl{\C O_1}) := \Omega_2 \otimes \Omega_3 \otimes \Omega_4/(\mathfrak{g}_2 \times \mathfrak{g}_3 \times \mathfrak{g}_4)\Omega_2 \otimes \Omega_3 \otimes \Omega_4 \longrightarrow I^{G_1}\left(\la[6,6] \boxtimes \mathcal{U}(\C O_1')\right) = \C B(\C O_1)
\end{equation}

Finally, by \cite[Theorem 3.5.1]{B5}, one has
$$
\Omega' \longrightarrow \mathcal{U}(\C O_0') \boxtimes \mathcal{U}(\C
O_1') = \ovl{X}_{O(6)}(-1,0,1) \boxtimes \ovl{X}_{Sp(14)}(-1,0,1;1,2,3,4)
$$

Hence Corollary \ref{cor:abstring}(a) gives the dual pair
correspondence of  $G_0 \times G_1$:
$$\Omega_1 \longrightarrow I^{G_1}\left(\la[6,6]^* \boxtimes \mathcal{U}(\C O_1')\right) \boxtimes I^{G_0}\left(\la[6,6] \boxtimes \mathcal{U}(\C O_0')\right)= \C B(\C O_1)^*\boxtimes \C B(\C O_0).$$
By combining the above map with \eqref{eq:o1}, along with the contraction arguments, one has
$$\Omega(\ovl{\C O_0}) = \Omega_1 \otimes \cdots \otimes \Omega_4/(\mathfrak{g}_1 \times \cdots \times \mathfrak{g}_4)\Omega_1 \otimes \cdots \otimes \Omega_4 \longrightarrow  \C B(\C O_0)$$
\end{example}

\begin{corollary} \label{cor:xbaroinduce}
Let $\C O$ be a classical nilpotent orbit. Then there is a $(\fg_c, K_c)-$equivariant map
$$\C E :\C B(\ovl{\C O}) \longrightarrow 
\C B(\C O) $$
with the spherical vector in its image.
\end{corollary}
\begin{proof}
The inclusion map
$$\mathcal{B}(\ovl{\C O}) := \Omega(\ovl{\C O})^{K^1/(K^1)^0}
	\hookrightarrow \Omega(\ovl{\C O})$$
has the trivial $K_c-$type in its image. Then the result follows directly
from the above inclusion and Proposition \ref{prop:xbaromap}.
\end{proof}

\begin{remark} \label{rmk:cyclic}
By Corollary \ref{cor:xbaroinduce}, $\C E$ maps the cyclic submodule of 
$\C B(\ovl{\C O})$ surjectively onto $\Xi(\C O) \subseteq \C B(\C O),$
the cyclic submodule of $\C B(\C O)$. In particular, $\Xi(\C O)$ gives a {\bf `lower bound'} on 
$\C B(\ovl{\C O})$.
In the next couple of sections, we will study $\Xi(\C O)$ in full detail, and conclude that it is also an {\bf `upper bound'}.
Consequently, we have $\C B(\ovl{\C O}) \cong \Xi(\C O)$, and $\C E$ is an injection as stated in the introduction (see Theorem \ref{thm:main}).
\end{remark}
%{\clrr The fact that $\C B(\ovl{\C O})$ is generated by the spherical
  %vector has been nagging at me. Is it obvious?}   {\clrblu We do {\bf NOT} need $\C B(\ovl{\C O})$ to be cyclic to 
	%get the lower bound any more...}

%%%%%%%%%%%%%%%
%%%%%%%%%%%%%%%
%%%%%%%%%%%%%%%
\section{The Cyclic Submodule of $\C B(\C O)$} \label{sec:cyclic}
In this section, we study the cyclic submodule $\Xi(\C O)$ of $\C B(\C O)$. The main result of this section is the following:
\begin{theorem} \label{thm:cyclic}
Let $\C O$ be a classical nilpotent orbit of the form 
$$\C O =  \begin{cases} (c_0 \geq c_1 \geq \dots \geq c_{2p} \geq c_{2p+1}) & \text{if}\ G = Sp(2n,\mathbb{C}),\\ 
( c_1 \geq c_2 \geq \dots \geq c_{2p} \geq c_{2p+1}) & \text{if}\ G = O(2n+\delta,\mathbb{C})
\end{cases}.$$ 
Consider the induced moudle
\begin{equation} \label{eq:gammao}
\Gamma(\C O) := \begin{cases}
I^G\left(\la[c_0,c_1] \boxtimes \cdots \boxtimes \la[c_{2p},c_{2p+1}]\right) & \text{if}\ G = Sp(2n,\mathbb{C}),   \\
I^G\left(\la[c_2,c_3] \boxtimes \cdots \boxtimes \la[c_{2p},c_{2p+1}] \boxtimes  \mathrm{triv}_{O(c_1)} \right) & \text{if}\ G = O(2n+\delta,\mathbb{C})
\end{cases}.
\end{equation}
Then the multiplicities of the diminutive $K_c-$types of $\Gamma(\C O)$ coincide with that of $\Xi(\C O)$. 
\end{theorem}

\begin{lemma} \label{lem:lem1}
Let $\C O$ be a nilpotent orbit. Then the module $\Gamma(\C O)$ 
defined in Theorem \ref{thm:cyclic} is cyclic.
\end{lemma}
\begin{proof}
Let $\la^0$ be dominant and a $W-$conjugate
of the infinitesimal character of $\Gamma(\C O)$. The module
$\Gamma(\C O)$ is the image of the intertwining operator 
$$I_1: X(\la^0) \longrightarrow X(w_1\la^0)$$
for some $w_1 \in W$ given in Equation (6.72) of \cite{B4}. 
For instance, in Equations (6.73) -- (6.77) of \cite{B4}, $w_1$ is given explicitly when $G$ is of Type $B$.
Since the spherical vector for the principal series at dominant
parameter is cyclic, and the induced module $\Gamma(\C O)$ is a homomorphic
image, $\Gamma(\C O)$ is cyclic as well. 
\end{proof}
%{\clrr I think the argument for Sp is an induction. Use reflections
  %by switching simple roots with parameter $\ge 0$ to move entries
  %$c_0,c_0-1,\dots,c_1, c_1-1,\dots ,1,1,0,0$ all the way to the
  %right. The image is an induced module with the unmoved coordinates
  %forming $GL(1)'$s and the rest give an induced $(0,\dots ,c_0)\boxtimes (0,\dots
  %,c_1-1)$. This induced module maps to the irreducible unipotent which
  %is also equal to the induced irreducible $(-c_1+1,\dots ,
  %c_0)\boxtimes (0)$. So we have a homomorphic image which is
  %induced from $GL'$s as desired. We can then move the long string
  %anywhere we want because viewing the $GL(1)-$coordinates as strings,
  %we get a nested set of strings, very strictly, containment. So move it to the
  %left, and when moving $c_2, c_2-1, \dots, 0$ across to the right, the image is the
  %induced $(-c_1+1,\dots ,c_0)$ which we wanted in the first place.
  %Apply induction to the remainder of the parameter. I may have the
  %$\la[c_0,c_1]$ messed up, but I think the idea is OK} 
%
%
%{\clrblu Indeed. I changed the numberings above so that it is in line with my notation. I think what you wrote above also shows up in Section 6.9 of \cite{B4}. 
\begin{example}
Consider $\C O = (8,6,4,4)$ in $G = Sp(22,\mathbb{C})$. Then 
$$\Gamma(\C O) := I^G(\la[8,6] \boxtimes \la[4,4]) = I^G((-2-101234)\boxtimes (-1012)),$$ 
and the dominant infinitesimal character is $\la^0 = (43222111100)$.

By applying suitable intertwining operators $\iota$ in \eqref{eq:iota} to the principal series representation
$$X(\la^0) = I^G((4) \boxtimes(3)\boxtimes(2)\boxtimes(2)\boxtimes(2)\boxtimes(1) \boxtimes(1)\boxtimes(1)\boxtimes(1)\boxtimes(1)\boxtimes(0)\boxtimes(0)),$$
Proposition \ref{p:glgl} implies that
$I^G((2) \boxtimes(2)\boxtimes(1)\boxtimes(1)\boxtimes(1)\boxtimes(0)\boxtimes(01234))$
is a homomorphic image of $X(\la^0)$. Similarly,
$$I^G((2) \boxtimes(1)\boxtimes(1)\boxtimes(012)\boxtimes(01234)) \overset{dm}{\rightarrow}
I^G((2) \boxtimes(1)\boxtimes(1)\boxtimes(01234)\boxtimes(012))$$
is a homomorphic image of $X(\la^0)$. Now reflect along the long roots and get
$$I^G((2) \boxtimes(1)\boxtimes(1)\boxtimes(01234)\boxtimes(012)) \rightarrow
I^G((2) \boxtimes(1)\boxtimes(1)\boxtimes(01234)\boxtimes(-2-10)).$$
By applying Proposition \ref{p:glgl} again to the strings $(01234)\boxtimes(-2-10)$, one has
$$I^G((2) \boxtimes(1)\boxtimes(1)\boxtimes(01234)\boxtimes(-2-10)) \rightarrow
I^G((2) \boxtimes(1)\boxtimes(1)\boxtimes(-2-101234)\boxtimes(0)).$$
Now the long string $\la[8,6] = (-2-101234)$ is nested with the strings on the left, hence
$$I^G((2) \boxtimes(1)\boxtimes(1)\boxtimes(-2-101234)\boxtimes(0)) \overset{dm}{\rightarrow}
I^G((-2-101234)\boxtimes(2) \boxtimes(1)\boxtimes(1)\boxtimes(0)).$$
By applying a similar argument on the strings $(2) \boxtimes(1)\boxtimes(1)\boxtimes(0)$, one can
conclude that
$$I^G((-2-101234)\boxtimes(2) \boxtimes(1)\boxtimes(1)\boxtimes(0)) \rightarrow
I^G((-2-101234)\boxtimes(-1012)).$$
is the homomorphic image of an intertwining operator.
\end{example}

\begin{proposition} \label{prop:main}
Let $\C O$ be a nilpotent orbit given in Theorem \ref{thm:cyclic} with $\tau(\C O)$ and 
$\C  O'$ are as in Definition \ref{def:roprime}. Then there is a non-trivial $(\fk g_c, K_c)-$morphism $\eta$ from $\Gamma(\C O)$ to the induced module
\begin{equation} \label{eq:glstrings}
\begin{aligned} 
&I^G\left( \la[c_{2r_1},c_{2r_1}] \boxtimes  \cdots \boxtimes  
\la[c_{2r_k},c_{2r_k}] \boxtimes 
\la[c_0',c_1'] \boxtimes  \cdots \boxtimes  \la[c_{2q}',c_{2q+1}']\right) \quad \text{or} \\ 
&I^G\left(  \la[c_{2r_1},c_{2r_1}] \boxtimes  \cdots \boxtimes  
\la[c_{2r_k},c_{2r_k}] \boxtimes 
\la[c_2',c_3'] \boxtimes  \cdots \boxtimes  \la[c_{2q}',c_{2q+1}']  \boxtimes \mathrm{triv}_{O(c_1')} \right)
\end{aligned}
\end{equation}
for $G = Sp(2n,\mathbb{C})$ or $O(2n+\delta,\mathbb{C})$ respectively. Moreover, $\eta$ is injective on the level of diminutive $K_c-$types.
\end{proposition}

We postpone the proof of the above proposition to the end of this section. 
Assuming its validity, we give a proof of Theorem \ref{thm:cyclic}.

\subsection{Proof of Theorem \ref{thm:cyclic}} \label{sec:proofcyclic}
Consider the special case when $\C O = \C O'$ first, so that $\Xi(\C O') = \C B(\C O') = \C U(\C O')$ is the (irreducible) unipotent representation.

By Lemma \ref{lem:lem1}, $\Gamma(\C O')$ is the image of the intertwining operator $I_1$. Then the intertwining operator 
$$I_2: X(w_1\la^0) \rightarrow X(-\la^0)$$
in Equation (6.72) of \cite{B4} maps $I_2(\Gamma(\C O')) = \ovl{X}(\la^0) = \C U(\C O')$ onto the irreducible subquotient.
Using the arguments of Equations (6.78) -- (6.80) of \cite{B4}, $I_2$ is injective on the level of diminutive $K_c'-$types, and hence
\begin{equation} \label{eq:surject3}
I_2|_{\Gamma(\C O')}: \Gamma(\C O') \xrightarrow{dm} \C U(\C O') = \Xi(\C O')
\end{equation}
as stated in the theorem.

For general $\C O$, apply induction in stages to \eqref{eq:surject3} and get a $(\fk g_c,K_c)-$morphism
\begin{align*}
\beta: I^G&\left( \la[c_{2r_1},c_{2r_1}] \boxtimes  \cdots \boxtimes  
\la[c_{2r_k},c_{2r_k}] \boxtimes \Gamma(\C O')\right)\\ 
\quad \quad  &\xrightarrow{dm}\ I^G\left( \la[c_{2r_1},c_{2r_1}] \boxtimes  \cdots \boxtimes  
\la[c_{2r_k},c_{2r_k}] \boxtimes \C U(\C O')\right) = \C B(\C O)
\end{align*}
whose domain is equal to \eqref{eq:glstrings}. Composing $\beta$ with $\eta$ in Proposition \ref{prop:main}, we have a surjection
%the arguments in Lemma \ref{lem:lem1}, one has a map
%\begin{equation} \label{eq:surject3}
%\begin{cases}
%\begin{aligned}
%&I^{G'}(\la[c_0',c_1'] \boxtimes  \cdots \boxtimes  \la[c_{2q}',c_{2q+1}'])\\ 
%&\longrightarrow \ovl{X}(\la[c_0',c_1'], \cdots, \la[c_{2q}',c_{2q+1}']) = \C U(\C O') \end{aligned} & \text{if}\ G' = Sp(2n',\mathbb{C}),\\
%\\
%\begin{aligned}
%I^{G'}&\left(\la[c_2',c_3'] \boxtimes  \cdots \boxtimes  \la[c_{2q}',c_{2q+1}'] \boxtimes \mathrm{triv}_{O(c_1')}\right)\\ 
%\longrightarrow &\ovl{X}\left(\la[c_2',c_3'],  \cdots,  \la[c_{2q}',c_{2q+1}'],  \la[-\delta,c_1']\right) = \C U(\C O')
%\end{aligned} & \text{if}\ G' = O(2n'+\delta,\mathbb{C})
%\end{cases}
%\end{equation}
%onto the irreducible quotient given by the intertwining operator $I_2: X(w_1\la^0) \rightarrow X(-\la^0)$ in Equation (6.72) of \cite{B4}. 
%By the arguments of Equations (6.78) -- (6.80) of \cite{B4}, $I_2$ is injective on the level of diminutive $K_c'-$types. Therefore, the maps in \eqref{eq:surject3}
%are isomorphisms on the level of diminutive $K_c'-$types. Applying induction in stages to the modules in \eqref{eq:surject3}, one can construct a map $\beta$ from
%the modules in \eqref{eq:glstrings} to $\C B(\C O)$, which is also an isomorphism on the level of diminutive $K_c-$types.
%
%\smallskip
\begin{equation} \label{eq:surject2}
\beta \circ \eta: \Gamma(\C O) \twoheadrightarrow \Xi(\C O) \quad (\subseteq \C B(\C O))
\end{equation}
since $\Gamma(\C O)$ and $\Xi(\C
O)$ are both cyclic. Now the result follows from the fact that $\eta$
and $\beta$ are injective on the level of diminutive $K_c-$types. \qed

\subsection{Proof of Proposition \ref{prop:main} for Symplectic Groups}
%As in the previous section, we present the proof of Theorem \ref{thm:cyclic} for even orbits.
%The proof for general orbits follows analogously by the discussions at the end of the previous section.
For $G = Sp(2n,\mathbb{C})$, one has to prove that
there is a map $\eta$ injective on the level of diminutive $K_c-$types:
\begin{align*}
\Gamma(\C O) :=\ &I^G\left(\la[c_0,c_1] \boxtimes \cdots \boxtimes \la[c_{2p},c_{2p+1}]\right) \\ 
\stackrel{\eta}{\longrightarrow}\ &I^G\left( \la[c_{2r_1},c_{2r_1}] \boxtimes  \cdots \boxtimes  
\la[c_{2r_k},c_{2r_k}] \boxtimes 
\la[c_0',c_1'] \boxtimes  \cdots \boxtimes  \la[c_{2q}',c_{2q+1}']\right)
\end{align*}
In fact, we will prove the above statement for $G = GL(n,\mathbb{C})$, so that the case
of $G = Sp(2n,\mathbb{C})$ follows from parabolic induction.

\medskip
We proceed by induction on the number of columns of $\C O$.  The result is automatic if $\C O = (c_0 \geq c_1)$. Now suppose $\C O = (c_0 \geq c_1 \geq c_2 \geq c_3)$. There are two cases:

\smallskip
\noindent (i) If $c_1 > c_2$, i.e. $\C O = \C O'$ with $\tau(\C O) = \phi$ is empty, then there is nothing to prove.

\noindent (ii) If $c_1 = c_2$, i.e. $\C O' = (c_0 \geq c_3)$ with $\tau(\C O) = \{1\}$, 
Consider the intertwining operator
$$I^{GL}\left(\la[c_0,c_3] \boxtimes  \la[c_2,c_2]\right) \stackrel{\iota}{\longrightarrow}
I^{GL}\left(\la[c_2,c_2] \boxtimes  \la[c_0,c_3]\right).$$
By Proposition \ref{p:glgl}(b), its image is equal to 
\begin{equation} \label{eq:incl2}
I^{GL}\left(\la[c_0,c_1] \boxtimes  \la[c_2,c_3]\right) \subseteq I^{GL}\left(\la[c_2,c_2] \boxtimes  \la[c_0,c_3]\right).
\end{equation}
So the proposition follows in this setting.

Suppose the result holds for $\C O_p = (c_0 \geq c_1  \geq \dots \geq c_{2p} \geq c_{2p+1})$, i.e.
\begin{equation} \label{eq:trueone}
\begin{aligned}
&I^{GL}\left(\la[c_0,c_1] \boxtimes  \cdots \boxtimes  \la[c_{2p},c_{2p+1}]\right) \\
\stackrel{\eta_p}{\longrightarrow} &\ I^{GL}\left(\la[c_{2r_1},c_{2r_1}] \boxtimes  \cdots \boxtimes  
\la[c_{2r_k},c_{2r_k}] \boxtimes 
\la[c_0',c_1'] \boxtimes  \cdots \boxtimes  \la[c_{2q}',c_{2q+1}']\right),
		\end{aligned}
\end{equation}
keeping in mind that $c_0' = c_0$ and $c_{2q+1}' = c_{2p+1}$.
Consider $\C O_{p+1} = (c_0 \geq c_1 \geq \dots  \geq c_{2p+2} \geq c_{2p+3})$.
By induction in stages, we have
\begin{align*}
&I^{GL}\left( 
\la[c_0,c_1] \boxtimes  \cdots \boxtimes  \la[c_{2p},c_{2p+1}] \boxtimes  \la[c_{2p+2},c_{2p+3}]
 \right)\\
\stackrel{\eta_p^+}{\rightarrow}
\ &I^{GL}\left(\la[c_{2r_1},c_{2r_1}] \boxtimes  \cdots \boxtimes  
\la[c_{2r_k},c_{2r_k}] \boxtimes 
\la[c_0',c_1'] \boxtimes  \cdots \boxtimes  \la[c_{2q}',c_{2q+1}'] \boxtimes  \la[c_{2p+2},c_{2p+3}] \right)\\
=\ &I^{GL}\left(\la[c_{2r_1},c_{2r_1}] \boxtimes  \cdots \boxtimes  
\la[c_{2r_k},c_{2r_k}] \boxtimes 
\la[c_0',c_1'] \boxtimes  \cdots \boxtimes  \la[c_{2q}',c_{2p+1}] \boxtimes  \la[c_{2p+2},c_{2p+3}]  \right),
\end{align*}
where the $\eta_p^+$ is obtained from \eqref{eq:trueone} and induction in stages. There are two cases:

\smallskip
\noindent (i) If $c_{2p+1} > c_{2p+2}$, then $p+1 \notin \tau(\C O)$ and we are done. 

\noindent (ii) If $c_{2p+1} = c_{2p+2}$, then $r_{k+1} = p+1 \in \tau(\C O)$, and one has
\begin{align*}
 \ &I^{GL}\left( \la[c_{2r_1},c_{2r_1}] \boxtimes  \cdots \boxtimes  
\la[c_{2r_k},c_{2r_k}] \boxtimes 
\la[c_0',c_1'] \boxtimes  \cdots \boxtimes  \la[c_{2q}',c_{2p+2}] \boxtimes  \la[c_{2p+2},c_{2p+3}]  \right)	\\
\subseteq \ &I^{GL}\left(\la[c_{2r_1},c_{2r_1}] \boxtimes  \cdots \boxtimes  
\la[c_{2r_k},c_{2r_k}] \boxtimes 
\la[c_0',c_1'] \boxtimes  \cdots \boxtimes  \la[c_{2p+2},c_{2p+2}] \boxtimes  \la[c_{2q}',c_{2p+3}]  \right) \\
\xrightarrow{\cong} \ &I^{GL}\left( \la[c_{2r_1},c_{2r_1}] \boxtimes  \cdots \boxtimes  
\la[c_{2r_k},c_{2r_k}] \boxtimes  \la[c_{2p+2},c_{2p+2}] \boxtimes  
\la[c_0',c_1'] \boxtimes  \cdots \boxtimes  \la[c_{2q}',c_{2p+3}] \right)\\
= \ &I^{GL}\left(\la[c_{2r_1},c_{2r_1}] \boxtimes  \cdots \boxtimes  
\la[c_{2r_k},c_{2r_k}] \boxtimes  \la[c_{2r_{k+1}},c_{2r_{k+1}}] \boxtimes  
\la[c_0',c_1'] \boxtimes  \cdots \boxtimes  \la[c_{2q}',c_{2p+3}]  \right),
\end{align*}
where the $\subseteq$ follows from \eqref{eq:incl2} and induction in stages. To see why the above $\overset{\cong}{\rightarrow}$ holds, note that $c_{2p+1} = c_{2p+2} \leq c_l'$ for all $l = 1, \dots, 2q$. So the strings $\la[c_{2j}',c_{2j-1}']$, $\la[c_{2p+2},c_{2p+2}]$
satisfy Condition (ii) of nestedness for all $j = 1, \dots, q$. Therefore, the proposition follows for $G = Sp(2n,\mathbb{C})$.

\subsection{Proof of Proposition \ref{prop:main} for Orthogonal Groups} Now study the case of $G = O(2n+\delta,\mathbb{C})$ for $\C O = (c_1 \geq c_2 \geq \dots \geq c_{2p+1})$. There are two cases:

\medskip
\noindent (i) If $c_1 \neq c_2$, i.e. $1 \neq r_1 \in \tau(\C O)$ and $c_1 = c_1'$. Then from above, there is a $(\fk g_c, K_c)-$morphism injective on the level of diminutive $K_c-$types:
\begin{align*}
I^{GL}&\left(\la[c_2,c_3] \boxtimes \dots \boxtimes \la[c_{2p},c_{2p+1}]\right) \\
&\stackrel{\eta_{GL}}{\longrightarrow}\ I^{GL}\left(\la[c_{2r_1},c_{2r_1}] \boxtimes  \cdots \boxtimes  
\la[c_{2r_k},c_{2r_k}] \boxtimes  \la[c_2',c_3'] \boxtimes  \cdots \boxtimes  \la[c_{2q}',c_{2q+1}']\right)
\end{align*}
and hence by induction in stages we have
\begin{align*}
\Gamma(\C O) =\ &I^{G}\left(\la[c_2,c_3] \boxtimes \dots \boxtimes \la[c_{2p},c_{2p+1}] \boxtimes  \mathrm{triv}_{O(c_1)}\right) \\ 
\stackrel{\eta_O}{\longrightarrow}
\ &I^{G}\left(\la[c_{2r_1},c_{2r_1}] \boxtimes  \cdots \boxtimes  
\la[c_{2r_k},c_{2r_k}] \boxtimes  \la[c_2',c_3'] \boxtimes  \cdots \boxtimes  \la[c_{2q}',c_{2q+1}'] \boxtimes  \mathrm{triv}_{O(c_1')}\right)
\end{align*}
to the induced module \eqref{eq:glstrings}, injective for all diminutive $K_c-$types.

\medskip
\noindent (ii) If $c_1 = c_2$, i.e. $1 = r_1 \in \tau(\C O)$. Then we have
\begin{align*}
&\Gamma(\C O) =\ I^{G}\left(\la[c_2,c_3] \boxtimes \dots \boxtimes \la[c_{2q},c_{2q+1}] \boxtimes  \mathrm{triv}_{O(c_1)}\right) \\ 
\stackrel{\eta_O}{\rightarrow}
\ &I^G\left(\la[c_{2r_2},c_{2r_2}] \boxtimes  \cdots \boxtimes  
\la[c_{2r_k},c_{2r_k}] \boxtimes  \la[c_2,c_1'] \boxtimes  \la[c_2',c_3'] \boxtimes \cdots \boxtimes  \la[c_{2q}',c_{2q+1}'] \boxtimes  \mathrm{triv}_{O(c_1)}\right) \\
=
\ &I^G\left(\la[c_{2r_2},c_{2r_2}] \boxtimes  \cdots \boxtimes  
\la[c_{2r_k},c_{2r_k}] \boxtimes  \la[c_2,c_1'] \boxtimes  \la[c_2',c_3'] \boxtimes \cdots \boxtimes  \la[c_{2q}',c_{2q+1}'] \boxtimes  \mathrm{triv}_{O(c_2)}\right) \\
\overset{\cong}{\rightarrow}
\ &I^G\left(\la[c_{2r_2},c_{2r_2}] \boxtimes  \cdots \boxtimes  
\la[c_{2r_k},c_{2r_k}] \boxtimes  \la[c_2',c_3'] \boxtimes  \cdots \boxtimes  \la[c_{2q}',c_{2q+1}'] \boxtimes  \la[c_2,c_1'] \boxtimes  \mathrm{triv}_{O(c_2)}\right) \\
%\hookrightarrow
%\ &I^G\left(\la[c_{2r_2},c_{2r_2-1}] \boxtimes  \cdots \boxtimes  
%\la[c_{2r_k},c_{2r_k-1}] \boxtimes  \la[c_2',c_3'] \boxtimes  \cdots \boxtimes  \la[c_{2q}',c_{2q+1}'] \boxtimes  \la[c_2,c_1'] \boxtimes  \la[0,c_1]\right) \\
\stackrel{(*)}{\rightarrow}
\ &I^G\left(\la[c_{2r_2},c_{2r_2}] \boxtimes  \cdots \boxtimes  
\la[c_{2r_k},c_{2r_k}] \boxtimes  \la[c_2',c_3'] \boxtimes  \cdots \boxtimes  \la[c_{2q}',c_{2q+1}'] \boxtimes  \la[c_2,c_2] \boxtimes  \mathrm{triv}_{O(c_1')}\right) \\
\overset{\cong}{\rightarrow}
\ &I^G\left( \la[c_2,c_2] \boxtimes \la[c_{2r_2},c_{2r_2}] \boxtimes  \cdots \boxtimes  
\la[c_{2r_k},c_{2r_k}] \boxtimes  \la[c_2',c_3'] \boxtimes  \cdots \boxtimes  \la[c_{2q}',c_{2q+1}'] \boxtimes  \mathrm{triv}_{O(c_1')}\right) \\
=
\ &I^G\left( \la[c_{2r_1},c_{2r_1}] \boxtimes \la[c_{2r_2},c_{2r_2}] \boxtimes  \cdots \boxtimes  
\la[c_{2r_k},c_{2r_k}] \boxtimes  \la[c_2',c_3'] \boxtimes  \cdots \boxtimes  \la[c_{2q}',c_{2q+1}'] \boxtimes  \mathrm{triv}_{O(c_1')}\right),
\end{align*}
where the $\overset{\cong}{\rightarrow}$'s above hold since the strings involved satisfy Condition (ii) of nestedness. On the other hand, 
the map $(*)$ above is defined as follows: For $a \geq b$, consider 
\begin{equation} \label{eq:orth}
I^{G'}\left(\la[a,b] \boxtimes  \mathrm{triv}_{O(a)}\right) \hookrightarrow
I^{G'}\left(\la[a,b] \boxtimes  \la[-\delta,a]\right) 
\stackrel{\iota}{\rightarrow}
I^{G'}\left( \la[-\delta,a] \boxtimes  \la[a,b]\right).
\end{equation}
We {\it claim} that the image of the map \eqref{eq:orth} is in $I^{G'}\left( \la[a,a] \boxtimes  \mathrm{triv}_{O(b)}\right)$. Indeed, by Proposition \ref{p:glgl}(b),
the image of $\iota$ is equal to $I^{G'}\left( \la[a,a] \boxtimes  \la[-\delta,b]\right)$. 
By restricting $\iota$ to the cyclic submodule $I^{G'}\left(\la[a,b] \boxtimes  \mathrm{triv}_{O(a)}\right)$, its image $\iota\left(I^{G'}\left(\la[a,b] \boxtimes  \mathrm{triv}_{O(a)}\right)\right) \subseteq I^{G'}\left( \la[a,a] \boxtimes  \la[-\delta,b]\right)$ must also be cyclic. Since
$$I^{G'}\left( \la[a,a] \boxtimes  \mathrm{triv}_{O(b)}\right) \subseteq I^{G'}\left( \la[a,a] \boxtimes  \la[-\delta,b]\right)$$
is a spherical submodule, one has
$$\iota\left(I^{G'}\left(\la[a,b] \boxtimes  \mathrm{triv}_{O(a)}\right)\right) \subseteq I^{G'}\left( \la[a,a] \boxtimes  \mathrm{triv}_{O(b)}\right) \subseteq I^{G'}\left( \la[a,a] \boxtimes  \la[-\delta,b]\right).$$ 

 So we are left to show \eqref{eq:orth} is an isomorphism on the level of diminutive $K_c'-$types.
This can be done by carefully tracing the intertwining operator
$\iota$ using the calculations in \cite[Section 6]{B4}. Details are given in Appendix A. 

\begin{example}
We continue with our example $\C O = (8,{\bf 6,6},4,2,0)$ in $\mathfrak{sp}(26,\mathbb{C})$ with $\C O' = (8,4,2,0)$. The strings 
$\la[6,6] = (-2\ldots3)$ and $\la[8,4] = (-1\dots4)$ are {\bf not} nested. So the intertwining operator
\begin{align*} I^G\left(\la[8,4] \boxtimes \la[6,6]\right) = &I^G\left((-1\ldots4)\boxtimes(-2\dots3)\right) \\
\stackrel{\iota}{\rightarrow} &I^G\left((-2\ldots3)\boxtimes(-1\dots4)\right) = I^G\left(\la[6,6] \boxtimes \la[8,4]\right)\end{align*}
has image
$$I^G\left((-2\ldots4)\boxtimes(-1\dots3)\right) = I^G\left(\la[8,6] \boxtimes \la[6,4]\right) \subseteq I^G\left(\la[6,6] \boxtimes \la[8,4]\right).$$
by Proposition \ref{p:glgl}(b). Using induction in stages, there is an inclusion 
$$\Gamma(\C O) := I^G\left(\la[8,6] \boxtimes \la[6,4] \boxtimes \la[2,0] \right) \subseteq I^G\left(\la[6,6] \boxtimes \la[8,4] \boxtimes \la[2,0] \right),$$
where $\Gamma(\C O)$ is cyclic. By the discussions in Section \ref{sec:proofcyclic}, we have:
$$I^G\left(\la[6,6] \boxtimes \la[8,4] \boxtimes \la[2,0] \right) \overset{dm}{\rightarrow} I^G\left(\la[6,6] \boxtimes \ovl{X}(\la[8,4],\la[2,0]) \right) = 
I^G\left(\la[6,6] \boxtimes \mathcal{U}(\C O') \right) =: \C B(\C O)$$
Restricting the above map to the cyclic submodule $\Gamma(\C O)$, one has a map between the cyclic modules
$\Gamma(\C O) \overset{dm}{\rightarrow}  \Xi(\C O)$ $(\subseteq \C B(\C O))$ as stated in Theorem \ref{thm:cyclic}.
\end{example}

%%%%%%%%%%%%%%%
%%%%%%%%%%%%%%%
%%%%%%%%%%%%%%%
\section{Proof of the Main Theorem}
We are now in the position to state and prove the main theorem of this
manuscript:
\begin{theorem} \label{thm:main}
Let $\C O$ be a classical nilpotent orbit. Then Equation \eqref{eq:CE} gives an isomorphism
onto the cyclic submodule $\Xi(\C O)$ of $\C B(\C O)$. 
\end{theorem}
\begin{proof}
Let $\C O$ be a classical nilpotent orbit as in Definition \ref{def:roprime}. By Remark \ref{rmk:cyclic}, 
the composition factors of the cyclic module $\Xi(\C O)$
must appear in $\C B(\ovl{\C O})$ as well. 
By Theorem \ref{thm:cyclic}, the multiplicities of diminutive $K_c-$types $V_{\beta}$ of 
$\Xi(\C O)$ is given by
\begin{equation} \label{eq:contradict1}
\begin{aligned}
\left[ \Xi(\C O): V_{\beta} \right] = & \left[ \Gamma(\C O): V_{\beta} \right]\\
= & \begin{cases} \left[I^G\left(\la[c_0,c_1] \boxtimes  \cdots \boxtimes  \la[c_{2p},c_{2p+1}]\right): V_{\beta} \right] & \text{if}\ G = Sp(2n,\mathbb{C}),\\ 
\left[I^G\left(\la[c_2,c_3] \boxtimes  \cdots \boxtimes  \la[c_{2p},c_{2p+1}] \boxtimes  \mathrm{triv}_{O(c_1)}\right): V_{\beta} \right] & \text{if}\ G = O(2n+\delta,\mathbb{C}) \end{cases} \\
= & \begin{cases} \left[\mathrm{Ind}_{GL(\frac{c_0+c_1}{2}) \times 
\dots \times GL(\frac{c_{2p}+c_{2p+1}}{2})}^G(\mathrm{triv}): V_{\beta} \right] & \text{if}\ G = Sp(2n,\mathbb{C}),\\ 
\left[\mathrm{Ind}_{GL(\frac{c_0+c_1}{2}) \times 
\dots \times GL(\frac{c_{2p}+c_{2p+1}}{2}) \times O(c_1)}^G(\mathrm{triv}): V_{\beta} \right] & \text{if}\ G = O(2n+\delta,\mathbb{C})
\end{cases}
\end{aligned}
\end{equation}

On the other hand, we have an inclusion of nilpotent varieties $\overline{\C O} \subseteq \overline{\C O^{\#}}$, where
\begin{equation} \label{eq:osharp}
\C O^{\#} := \begin{cases} \left(\frac{c_0+c_1}{2} \geq \frac{c_0+c_1}{2} \geq \dots \geq \frac{c_{2p}+c_{2p+1}}{2} \geq \frac{c_{2p}+c_{2p+1}}{2}\right) & \text{if}\ G = Sp(2n,\mathbb{C}),\\
\left(c_1 \geq \frac{c_2+c_3}{2} \geq \frac{c_2+c_3}{2} \geq \dots \geq \frac{c_{2p}+c_{2p+1}}{2} \geq \frac{c_{2p}+c_{2p+1}}{2}\right) &  \text{if}\ G = O(2n+\delta,\mathbb{C})
\end{cases}
\end{equation}
Note that $\overline{\C O^{\#}}$ is normal and is a Richardson orbit induced from 
$$L = \begin{cases} GL(\frac{c_0+c_1}{2}) \times 
\dots \times GL(\frac{c_{2p}+c_{2p+1}}{2}) & \text{if}\ G = Sp(2n,\mathbb{C}),\\
 GL(\frac{c_2+c_3}{2}) \times 
\dots \times GL(\frac{c_{2p}+c_{2p+1}}{2}) \times O(c_1) & \text{if}\ G = O(2n+\delta,\mathbb{C}) \end{cases}$$ 
Hence the multiplicities of all $K_c-$types of $R(\ovl{\C O})$ is {\bf bounded above} by
\begin{equation} \label{eq:contradict2}
\left[ R(\ovl{\C O}): V_{\beta} \right] \leq 
\left[R(\overline{\C O^{\#}}) : V_{\beta} \right],
\end{equation}
where
\begin{equation} \label{eq:contradict3}
\begin{aligned}
\left[R(\overline{\C O^{\#}}) : V_{\beta} \right] &= \left[R(\C O^{\#}) : V_{\beta} \right]  \\
&= \begin{cases} \left[\mathrm{Ind}_{GL(\frac{c_0+c_1}{2}) \times 
\dots \times GL(\frac{c_{2p}+c_{2p+1}}{2})}^G(\mathrm{triv}): V_{\beta} \right] & \text{if}\ G = Sp(2n,\mathbb{C}),\\
\left[\mathrm{Ind}_{GL(\frac{c_2+c_3}{2}) \times 
\dots \times GL(\frac{c_{2p}+c_{2p+1}}{2}) \times O(c_1)}^G(\mathrm{triv}): V_{\beta} \right] & \text{if}\ G = O(2n+\delta,\mathbb{C}) \end{cases}
\end{aligned}
\end{equation}

Suppose on the contrary that $\C B(\ovl{\C O})$ is not isomorphic to $\Xi(\C O)$. Then it must
have at least one extra composition factor other than those appearing $\Xi(\C O)$. Let $\pi$ be one of such
factor; its lowest $K_c-$type, $V_{\gamma}$, must be diminutive by Proposition \ref{prop:parastructure}.
So one has
\begin{equation} \label{eq:contradict}
[R(\ovl{\C O}):V_{\gamma}] = [\C B(\ovl{\C O}):V_{\gamma}] 
\geq [\pi \oplus \Xi(\C O) : V_{\gamma}] 
> [\Xi(\C O) : V_{\gamma}] 
\stackrel{\eqref{eq:contradict1}, \eqref{eq:contradict3}}{=} [R(\ovl{\C O^{\#}}) : V_{\gamma}],
\end{equation}
contradicting Equation \eqref{eq:contradict2}. Therefore, $\C B(\ovl{\C O})$ must have the same
composition factors as $\Xi(\C O)$. 
\end{proof}

\begin{corollary} \label{cor:discrepancy}
Let $\C O$ be a classical nilpotent orbit. The diminutive $K_c-$type multiplicity of $R(\ovl{\C O})$ is equal to that of 
$R(\ovl{\C O^{\#}})$, where $\ovl{\C O^{\#}} \supset \ovl{\C O}$ is given in Equation \eqref{eq:osharp}.
Moreover, $\ovl{\C O}$ is not normal iff $R(\C O)$ and $R(\ovl{\C O})$
have different multiplicities on the level of diminutive $K_c-$types.
\end{corollary}
\begin{proof}
The statement on diminutive $K_c-$type multiplicities are precisely Theorem \ref{thm:main}.
For the second statement, let $\C  O = (c_0 \geq c_1 \geq \dots \geq c_{2p} \geq c_{2p+1})$
for $G = Sp(2n,\mathbb{C})$, and $\C  O = (c_1 \geq c_2 \geq \dots \geq c_{2p} \geq c_{2p+1})$ for $G = O(2n+ \delta,\mathbb{C})$.
By Theorem \ref{thm:main} and \cite{B5}, the diminutive $K_c-$type
multiplicities of $R(\ovl{\C O})$ and $R(\C O)$ are given by the induced modules
\begin{equation} \label{eq:indro}
\mathrm{Ind}_{\prod_{j= 0}^p GL(\frac{c_{2j}+c_{2j+1}}{2})}^G(\mathrm{triv})\ \ \text{and}\ \ 
\mathrm{Ind}_{\prod_{i \in \mathcal{R}} GL(c_{2i}) \times \prod_{j= 0}^q GL(\frac{c_{2j}'+c_{2j+1}'}{2})}^G(\mathrm{triv})
\end{equation}
for $G = Sp(2n,\mathbb{C})$, and
\begin{equation} \label{eq:indro2}
\mathrm{Ind}_{\prod_{j= 1}^p GL(\frac{c_{2j}+c_{2j+1}}{2}) \times O(c_1)}^G(\mathrm{triv})\ \ \text{and}\ \ 
\mathrm{Ind}_{\prod_{i \in \mathcal{R}} GL(c_{2i}) \times \prod_{j= 1}^q GL(\frac{c_{2j}'+c_{2j+1}'}{2}) \times O(c_1')}^G(\mathrm{triv})
\end{equation}
for $G = O(2n+ \delta,\mathbb{C})$ respectively. The two induced modules have the same multiplicities iff
they are induced from the same Levi type, which occurs precisely when
$\C O$ is of the form given in Theorem \ref{thm:KP}. 
\end{proof}

\begin{example}
We study the orbit $\mathcal{O} = (6,4,4,2,2,0)$ in $\mathfrak{sp}(18,\mathbb{C})$. Then
the induced modules in Equation \eqref{eq:indro} are
\begin{equation}
\mathrm{Ind}_{GL(5) \times 
GL(3) \times GL(1)}^G(\mathrm{triv})\ \ \text{and}\ \ 
\mathrm{Ind}_{GL(4) \times GL(2) \times GL(3)}^G(\mathrm{triv}).
\end{equation}
By Frobenius reciprocity, the discrepancy of $K_c-$type multiplicities between $R(\ovl{\C O})$ and $R(\C O)$ occurs as early as $V_{(1,1,1,1,0,0,0,0,0)}$. 

By the Appendix, the seven possible composition factors of $\C B(\ovl{\C O})$ and $\C B(\C O)$ are
\begin{align*}
&\pi_1 := \mathrm{Ind}_{GL(2) \times GL(4) \times Sp(6)}^G(\det \boxtimes \det \boxtimes\ \mathrm{triv}),\quad \pi_2^{\pm} := \mathrm{Ind}_{GL(4) \times Sp(10)}^G(\det \boxtimes\ \mathcal{U}(6,4; \pm)), \\
&\pi_3^{\pm} := \mathrm{Ind}_{GL(2) \times Sp(14)}^G(\det \boxtimes\ \mathcal{U}(6,6,2;\pm)),\quad \pi_4^{\pm} := \mathcal{U}(6,6,3,3; \pm).
\end{align*}
where $\mathcal{U}(\mathcal{P};\epsilon)$ are the special unipotent representations attached to
the (special) orbit $\C P$. All these factors are unitary.

By studying closely the expression of $\C B(\C O)$ in Definition \ref{def:deform}, one can check that
all the seven factors appears in the composition series of $\C B(\C O)$ with multiplicity one.
As for $\C B(\ovl{\C O}) \cong \Xi(\C O)$, one can obtain the composition series of $\Xi(\C O)$ by studying that
of $\Gamma(\C O) = I^G(\la[6,4], \la[4,2], \la[2,0])$ in Equation \eqref{eq:gammao}. In our subsequent work, we show that the composition factors
of $\C B(\ovl{\C O})$ are precisely $\pi_1$, $\pi_2^-$, $\pi_3^+$ and $\pi_4^+$. 
\end{example}

%%%%%%%%%%%%%%%
%%%%%%%%%%%%%%%
%%%%%%%%%%%%%%%

\appendix
\section{Image of an Intertwining Operator}
Let $G = O(2n+\delta,\mathbb{C})$. In this section, we will prove that the map 
\begin{equation}  \label{eq:inciota}
\iota|_{I^{G}\left(\la[a,b] \boxtimes  \mathrm{triv}_{O(a)}\right)} : I^{G}\left(\la[a,b] \boxtimes  \mathrm{triv}_{O(a)}\right) \longrightarrow I^{G}\left( \la[a,a] \boxtimes  \mathrm{triv}_{O(b)}\right).
\end{equation}
defined in \eqref{eq:orth} is injective on the level of diminutive $K_c-$types, where $\iota$ is the intertwining operator 
\begin{equation*}
I^{G}\left(\la[a,b] \boxtimes \la[-\delta,a] \right) \stackrel{\iota}{\longrightarrow} I^{G}\left( \la[-\delta,a] \boxtimes  \la[a,b]\right).
\end{equation*}
Let $p:= \frac{1}{2}(a+b)$, $q := \frac{1}{2}(a-\delta)$, $M = GL(p,\mathbb{C}) \times GL(q,\mathbb{C})$ so that $n=p+q$, and $w \in W(M) = S_p \times S_q$ be the involution swapping the first $p$ coordinates with the last $q$ coordinates. Then the above intertwining operator can be rewritten in the notations of \cite{B4} as
$$\iota(w): X_M((\nu_1)(\nu_2)) \longrightarrow X_M((\nu_2)(\nu_1))$$
with $\nu_1 = a-b+1$ and $\nu_2 = \frac{a+\delta-2}{2}$. 

\medskip
In \cite{B4}, the signatures of $\iota(w)$ for the isotypic components of all diminutive (or more generally {\it relevant}) $K_c-$types are explicitly computed. More precisely, let $V_i := \wedge^{2i}\mathbb{C}^{2n'+\delta}$ be a diminutive $K_c-$type. Consider
$$V_i \quad \longleftrightarrow\quad \Sigma_i := (n-i) \times (i),$$
where $\Sigma_i$ is the irreducible $W(G)-$module obtained by the zero weight space of $V_i^*$.
Then the image of $\iota(w)$ on the $V_i-$isotypic component of $X_M((\nu_1)(\nu_2))$ can be obtained by studying the operator
\begin{equation} \label{eq:iw}
R_{i}(w): (\Sigma_i)^{W(M)} \longrightarrow (\Sigma_i)^{W(w\cdot M)}.
\end{equation}
given in Equation (5.3) of \cite{B4}.

To check whether \eqref{eq:inciota} is injective for diminutive $K_c-$types, consider the inclusion $(\Sigma_i)^{W(GL(p) \times O(2q+\delta))} \subseteq (\Sigma_i)^{W(M)}$, and check whether the restriction of \eqref{eq:iw} to $(\Sigma_i)^{W(GL(p) \times O(2q+\delta))}$
is an injection for all $i$.

We now compute $(\Sigma_i)^{W(GL(p) \times O(2q+\delta))}$ explicitly. Let 
$$\{e_1,\ \ldots,\ e_n,\ f_1,\ \ldots,\ f_n,\ \delta\}$$
be an isotropic basis of $(\mathbb{C}^{2n+\delta})^*$. Then the $\Sigma_i$ is spanned by
$$\{\theta_{k_1} \wedge \ldots \wedge \theta_{k_i} |\ 1 \leq k_1 < \ldots < k_i \leq n \},$$
where $\theta_1 := e_1 \wedge f_1,\ \ldots,\ \theta_n := e_n \wedge f_n.$ It is easy to see that 
$(\Sigma_i)^{W(GL(p) \times O(2q+\delta))}$ is one dimensional for $1 \leq i \leq p$, spanned by the single vector
$${\bf v}_i := \sum_{1 \leq k_1 < \dots < k_i \leq p} \theta_{k_1} \wedge \dots \wedge \theta_{k_i}$$
So one needs to check $R_{i}(w)({\bf v}_i) \neq 0$.

\medskip
Recall the calculation of $R_i(w)$ around Equations (6.12) -- (6.14) of \cite{B4}: Consider the restriction of the $W(G)-$module
$$(n-i) \times (i)|_{S_n} = (n) \oplus (n-1,1) \oplus \dots \oplus (n-i,i)$$
into $W(GL(n,\mathbb{C})) = S_n-$modules. Since $\Sigma_i \cong (n-i) \times (i)$ as $W(G)-$modules, one has the decomposition
\begin{equation} \label{eq:decompose}
(\Sigma_i)^{W(M)} = \left(\Sigma_{i,(n)}\right)^{W(M)} \oplus \dots \oplus  \left(\Sigma_{i,(n-i,i)}\right)^{W(M)}
\end{equation}
where $\Sigma_{i,(n-k,k)}$ is the $(n-k,k)-$isotypic component of the restricted module $\Sigma_i$ to $S_n$.
For all ${\bf x}_i \in (\Sigma_i)^{W(M)}$, write
\begin{equation} \label{eq:decompose2}
{\bf x}_i = {\bf x}_{i,0} + \dots + {\bf x}_{i,i}, \quad \quad {\bf x}_{i,k} \in (\Sigma_{i,(n-k,k)})^{W(M)}
\end{equation}
under the decomposition \eqref{eq:decompose}, then
$$R_{i}(w)({\bf x}_i) = \alpha_0{\bf x}_{i,0} + \dots + \alpha_i{\bf x}_{i,i},$$
where $\alpha_k := r_{\sigma(n-k,k)}(p,q,\nu_1,\nu_2) \in \mathbb{R}$ is given in Equation (6.15) of \cite{B4}. In particular, $\alpha_0 = 1$.

\medskip
Now we are ready to compute $r_{i}(w)({\bf v}_i)$. Under the decomposition \eqref{eq:decompose2} ${\bf v}_i = {\bf v}_{i,0} + \dots + {\bf v}_{i,i}$ of ${\bf v}_i$, 
$${\bf v}_{i,0} = \frac{1}{n!}\sum_{w \in S_n} \sum_{1 \leq k_1 < \dots < k_i \leq p} w\cdot (\theta_{k_1} \wedge \dots \wedge \theta_{k_i})
= \frac{i!(n-i)!}{n!}\sum_{1 \leq \ell_1 < \dots < \ell_i \leq n} \theta_{\ell_1} \wedge \dots \wedge \theta_{\ell_i} \neq {\bf 0}.$$
Therefore, for all $1 \leq i \leq p$,
$$r_{i}(w)({\bf v}_i) = \alpha_0{\bf v}_{i,0} + \dots + \alpha_i{\bf v}_{i,i} = {\bf v}_{i,0} + \dots + \alpha_i{\bf v}_{i,i} \neq {\bf 0}$$
since the vectors $\{{\bf v}_{i,k} |\ 0 \leq k \leq i\}$ are linearly independent, and consequently \eqref{eq:inciota} is an injection for all diminutive $K_c-$types.

\section{Composition Factors of $\C B(\C O)$}
In this section, we list the possible Langlands 
parameters that {\it can} appear in the composition series of $\C B(\C O)$ and $\mathcal{B}(\ovl{\C O})$
for even orbits of Type $C$. The results can be generalized to all other classical orbits.

\smallskip
For any orbit $\C O = (c_0 \geq c_1 \geq \dots \geq c_{2p} \geq c_{2p+1})$, the infinitesimal character attached to $\C B(\C O)$ and $\mathcal{B}(\ovl{\C O})$ is
\begin{equation}
  \label{eq:infl}
  c_{2i}\longrightarrow (\frac{c_{2i}}{2},\dots ,1);\ \ 
  c_{2i+1}\longrightarrow (\frac{c_{2i+1}}{2}-1,\dots ,1,0).
\end{equation}
This is implicit from the constructions in Section \ref{sec:indprinciple}. Moreover, it also matches the dual pair correspondence in, 
for example, \cite{Pz}.

Let $\la$ be as in Equation \eqref{eq:infl}. 
There is a unique maximal primitive ideal $I(\la)\subset U(\fk g)$
with a given infinitesimal character $\la.$ This
determines a nilpotent orbit $\CO_\la$ such that any admissible
irreducible $(\fk g_c, K_c)-$module will have at least 
$\ovl{\CO_\la}$ as its associated variety. In
particular, the spherical irreducible module with infinitesimal
character $\la$  has associated variety precisely the closure of $\CO_\la$
(see \cite{BV2} for details). 
The composition factors of $\C B(\C O)$ and $\mathcal{B}(\ovl{\C O})$ 
are irreducible $(\fg_c, K_c)-$modules with infinitesimal 
character $\lambda$ and associated varieties
contained in  $\overline{\C O}$ and containing 
$\overline{\C O_{\lambda}}$.

\subsection{Definition of $Norm(\C O)$}
In this section, we describe all nilpotent orbits $\C P$ lying between
$\overline{\C O_{\lambda}}$ and $\overline{\C O}$.
\begin{definition} \label{def:funddegeneration}
Let $(b_0 \geq b_1 = b_2 \geq b_3)$ be four positive integers of the same parity. 
A {\bf fundamental degeneration} is defined by:

$\bullet\ (b_0>b_1=b_2>b_3) \rightarrow (b_0, b_1+2, b_2-2, b_3)$.

$\bullet\ (b_0=b_1=b_2>b_3)  \rightarrow (b_0+1, b_1+1, b_2-2, b_3)$.

$\bullet\ (b_0>b_1=b_2=b_3)  \rightarrow (b_0, b_1+2, b_2-1, b_3-1)$.

$\bullet\ (b_0=b_1=b_2=b_3)  \rightarrow (b_0+1, b_1+1, b_2-1, b_3-1)$.
\end{definition}
In the first two cases, we omit the columns $b_2-2,b_3$ 
if both terms are equal to zero. Note that when $b_0 > b_1$, 
the size of $b_0$ remains unchanged after degeneration. Similarly, 
if $b_2 > b_3$, the size of $b_3$ is the same after degeneration.
\begin{definition} 
Let $\C O = (c_0 \geq c_1 \geq \dots \geq c_{2p} \geq c_{2p+1})$ be an even orbit. We construct a collection of orbits as follows:

\medskip
\noindent (1)\ For each $\dots c_{2i} \geq c_{2i+1} = c_{2i+2} = \dots = c_{2j-1} = c_{2j} \geq c_{2j+1} \dots$ 
appearing in $\C O$, perform fundamental degeneration on the columns 
$c_{2i} \geq c_{2i+1} = c_{2j} \geq c_{2j+1}$
and get a new orbit:
$$(c_0 \geq \dots \geq c_{2i}'\geq c_{2i+1}' \geq c_{2i+2} = \dots = c_{2j-1}\geq c_{2j}'\geq c_{2j+1}'\geq \dots\geq c_{2p+1}).$$
\noindent (2)\ For each new orbit obtained in Step (1), repeat Step (1) on them until
  there are no more $c_{2j+1} = c_{2j+2}$'s.
		
\medskip
Denote the collection of all such orbits by $Norm(\C O)$. They are precisely the orbits 
between $\overline{\C O_{\lambda}}$ and $\overline{\C O}$.
\end{definition}

\begin{example}
Let $\C O = (8\geq 6\geq 6\geq 4\geq 4\geq 2\geq 2\geq 0)$. Then the $\C P \in Norm(\C O)$ are given by
\begin{equation} \label{eq:8664422}
  \begin{matrix}
    & &(8 6  6 4  4 2  2  0)&\\
		&&&\\
&(8 8 4 4  4 2  2  0)&(8 6  6  6 2 2  2  0)&(8 6  6 4  4  4)\\
&&&\\
&(8 8 5  5 2 2  2  0)&(8  8 4 4  4  4)&(8 6  6  6 3  3)\\
&&&\\
&         &(8  8 5  5 3  3)
  \end{matrix}.
	\end{equation}
\end{example}

The following proposition gives an explicit construction of 
the irreducible modules $\mathcal{M}(\C P, \ep)$ with infinitesimal character given by \eqref{eq:infl} 
and associated variety $AV(\mathcal{M}(\C P, \ep)) = 1\cdot \ovl{\C P}$ for all ${\C P} \in Norm(\C O)$. 

\begin{proposition} \label{prop:parastructure}
Let $\C O$ be an even orbit. For each ${\C P} = (d_1 \geq \dots \geq d_{2s+1}) \in Norm(\C O)$,
let 
$$\C \tau_0(\C P) := \{i |\ d_{2i-1} = d_{2i}\ \text{is even}\}$$
and $\C P^*$ is obtained from $\C P$ by removing the columns $d_{2i-1} = d_{2i}$ for $i \in \C \tau_0(\C P)$.
Consider the induced modules 
\begin{equation}\label{eq:parameter}
\mathcal{M}(\C P, \ep) := \mathrm{Ind}_{\prod_{i \in \C \tau_0(\C P)} GL(d_{2i}) \times G^*}^G\left( \det \boxtimes \dots \boxtimes \det \boxtimes\ \mathcal{U}(\C P^*;\epsilon)\right),
\end{equation}
where each $\epsilon \in \overline{A}(\C P^*)^{\vee}$ is an irreducible representation of
the {\it Lusztig's quotient group} $\overline{A}(\C P^*)$ of $\C P^*$, and $\mathcal{U}(\C P^*;\epsilon)$ are the special unipotent representations attached to the special orbit $\C P^*$.

Then all $\mathcal{M}(\C P, \ep)$ are irreducible having
associated variety $\overline{{\C P}}$. Moreover, this exhausts all
irreducible representations with infinitesimal character given by Equation \eqref{eq:infl}
and associated variety $\overline{{\C P}}$.
\end{proposition}
We will prove the proposition in the next subsection.

\medskip
Since the modules $\C B(\C O)$ and $\C B(\ovl{\C O})$
constructed in Section \ref{sec:indprinciple} have infinitesimal character
given by Equation \eqref{eq:infl} and associated variety $\ovl{\C O}$,
the last paragraph of the above Proposition says all composition 
factors of $\C B(\C O)$ and $\C B(\ovl{\C O})$ must be of the form
given by Equation \eqref{eq:parameter}. 

\smallskip
By construction, all the representations in Equation \eqref{eq:parameter} are unitary. Furthermore, the lowest $K_c^*-$types of
the unipotent representations $\mathcal{U}(\C P^*;\epsilon)$ are diminutive, 
so the induced modules in Equation \eqref{eq:parameter} also have diminutive lowest $K_c-$types. 
This observation is essential in the proof of Theorem \ref{thm:main}.

\begin{example}
Let $\C O = (8\geq 6\geq 6\geq 4\geq 4\geq 2\geq 2\geq 0)$ as above. Then the induced modules in Equation \eqref{eq:parameter} are:
\small{
\[
  \begin{matrix}
    & &I_{GL(6,4,2) \times G_8^*}^G(\det^{\boxtimes 3} \boxtimes\ \mathcal{U}(80))&\\
		&&&\\
&I_{GL(4,2) \times G^*_{20}}^G(\det^{\boxtimes 2} \boxtimes\ \mathcal{U}(8840;\pm))&I_{GL(6,2) \times G^*_{16}}^G(\det^{\boxtimes 2} \boxtimes\ \mathcal{U}(8620;\pm))&I_{GL(6,4) \times G^*_{12}}^G(\det^{\boxtimes 2} \boxtimes\ \mathcal{U}(84;\pm))\\
&&&\\
&I_{GL(2)\times G^*_{28}}^G(\det \boxtimes\ \mathcal{U}(885520;\pm))&I_{GL(4) \times G^*_{24}}^G(\det \boxtimes\ \mathcal{U}(8844;\pm,\pm))&I_{GL(6) \times G^*_{20}}^G(\det \boxtimes\ \mathcal{U}(8633;\pm))\\
&&&\\
&         &\mathcal{U}(885533;\pm)
  \end{matrix}.
\]}
where $GL(n_1, \dots, n_k) := GL(n_1) \times \dots \times GL(n_k)$ and $G^*_m = Sp(m,\mathbb{C})$. In particular, there are
17 irreducible representations in total, all of which are unitary.	
\end{example}

\subsection{Proof of Proposition \ref{prop:parastructure}} Since $\C O$ is even,
the infinitesimal character $\lambda$ in Equation \eqref{eq:infl} is integral. 

Given the integral infinitesimal character $(\lambda,\lambda)$ in 
Equation \eqref{eq:infl}, we study the {\bf left cone representation} 
corresponding to the orbit $\C O_{\lambda}$ as described in
the beginning of Section 4:
$$\ovl{V}^L(w_0w_{\lambda}) = \mathrm{Ind}_{W(\lambda)}^W(\mathrm{triv}),$$
where $W(\lambda)$ is the largest Weyl Levi subgroup of $W$ fixing $\lambda$,
and $w_{\lambda}$ is the longest element in $W(\lambda)$.  

For any special orbit $\C P$, the number of irreducible representations 
with infinitesimal character $(\lambda,\lambda)$ and associated variety $\ovl{\C P}$ is given by 
$$[\ovl{V}^L(w_0w_{\lambda}): V(\C P)] = [\mathrm{Ind}_{W(\lambda)}^W(\mathrm{triv}):V(\C P)],$$
where $V(\C P)$ is the {\bf left cell representation} corresponding to the orbit $\C P$.
In particular, each $V(\C P)$ contains $|\ovl{A}(\C P)|$ distinct irreducible representations
of $W$ (see \cite{BV2}, Proposition 5.28 for instance).

The decomposition of $\ovl{V}^L(w_0w_{\lambda}) = \mathrm{Ind}_{W(\lambda)}^W(\mathrm{triv})$ into irreducible
representations can be easily computed using the Robinson-Schensted algorithm. In particular, 
one can check that for any even $\C O$ and any $\C P \in Norm(\C O)$, each irreducible 
factor of $V(\C P)$ appears in $\mathrm{Ind}_{W(\lambda)}^W(\mathrm{triv})$ exactly once.
\begin{example}
Consider $\C O = (6\geq4\geq4\geq2\geq2\geq0)$ with $\lambda = (322111100)$. Then 
$$Norm(\C O) = \{ \C O, \C P_1 := (6\geq4\geq4\geq4), \C P_2 := (6\geq 6\geq 2\geq 2\geq 2\geq 0), \C O_{\lambda} = (6\geq 6\geq 3\geq 3)\}$$
and $W(\lambda) = W(A_0) \times W(A_1) \times W(A_3) \times W(C_2)$. 

The left cell representation for each $\C P \in Norm(\C O)$ is given by the following:
\begin{itemize}
\item $\C O$: $(21 \times 321)$.
\item $\C P_1$: $(22 \times 32)$, $(43 \times 11)$.
\item $\C P_2$: $(31 \times 311)$, $(41 \times 211)$.
\item $\C O_{\lambda}$: $(32 \times 31)$, $(42 \times 21)$.
\end{itemize}
Note that the Young diagrams are all described in terms of columns, i.e. \newline
$(43 \times 11)$ = \begin{tabular}{cc|c|c|cccc}
\cline{3-4} \cline{6-7} 
\multirow{4}{*}{} &  &  &  & \multicolumn{1}{c|}{} & \multicolumn{1}{c|}{} & \multicolumn{1}{c|}{} & \tabularnewline
\cline{3-4} \cline{6-7} 
 & \multirow{2}{*}{$($} &  &  & \multirow{2}{*}{,} &  &  & \multirow{2}{*}{$)$}\tabularnewline
\cline{3-4} 
 &  &  &  &  &  &  & \tabularnewline
\cline{3-4} 
 &  &  & \multicolumn{1}{c}{} &  &  &  & \tabularnewline
\cline{3-3} 
\end{tabular}.

One can easily show that each irreducible representation above shows up exactly once in
$\mathrm{Ind}_{W(\lambda)}^W(\mathrm{triv})$. In other words, there
is only one way to fill up the Young diagrams into semi-standard Young tableaux using the 
coordinates of $\lambda$ with the first row of the left Young diagram filled with zeros. 
For example, the only way to fill up $(43 \times 11)$ is given by \begin{tabular}{cc|c|c|cccc}
\cline{3-4} \cline{6-7} 
\multirow{4}{*}{} &  & $0$ & $0$ & \multicolumn{1}{c|}{} & \multicolumn{1}{c|}{$1$} & \multicolumn{1}{c|}{$1$} & \tabularnewline
\cline{3-4} \cline{6-7} 
 & \multirow{2}{*}{$($} & $1$ & $1$ & \multirow{2}{*}{,} &  &  & \multirow{2}{*}{$)$}\tabularnewline
\cline{3-4} 
 &  & $2$ & $2$ &  &  &  & \tabularnewline
\cline{3-4} 
 &  & $3$ & \multicolumn{1}{c}{} &  &  &  & \tabularnewline
\cline{3-3} 
\end{tabular}.

For an arbitrary orbit $\C O$ with even columns, there is a systematic way of describing 
the left cell representations for all $\C P \in Norm(\C O)$ (c.f. \cite[Chapter 4]{Lu})
and the same results hold. We omit the details here.
\end{example}

Consequently, the number of irreducible modules with infinitesimal character $(\lambda,\lambda)$ and
associated variety $\ovl{\C P}$ for each $\C P \in Norm(\C O)$ is equal to $|\ovl{A}(\C P)|$.
By our definition of fundamental degeneration and $Norm(\C O)$,
one can check that all $\C P \in Norm(\C O)$ are special orbits. Moreover,
the number of irreducible representations of the form $\mathcal{M}(\C P, \ep)$ in Proposition \ref{prop:parastructure}
is equal to $|\ovl{A}(\C P^*)|$ which is equal to $|\ovl{A}(\C P)|$ by the construction of $\C P^*$ from $\C P$. 
So we are left to show that the induced module $\mathcal{M}(\C P, \ep)$
in Equation \eqref{eq:parameter} is irreducible.

\smallskip
Let $\ovl{X}(\C P, \ep)$ be the irreducible subquotient of $\mathcal{M}(\C P, \ep)$ having the same lowest $K_c-$type. 
Then the lowest $K_c-$type of both modules are diminutive, and hence all
diminutive $K_c-$types of $\ovl{X}(\C P, \ep)$ are \emph{bottom-layer} in the sense
of \cite{SV}. Therefore, one can check, by \cite[Proposition 2.7]{B1} for instance, that
\begin{equation} \label{eq:last}
\mathcal{M}(\C P, \epsilon) \approx  \ovl{X}(\C P, \ep),
\end{equation}
i.e. they have the same diminutive $K_c-$type multiplicities.
In order to prove Proposition \ref{prop:parastructure}, one needs to check the $\approx$ in \eqref{eq:last}
is indeed an equality. 

\smallskip
We proceed by induction on the closure ordering of the orbits in $Norm(\C O)$: 
Consider the smallest orbit $\C O_{\lambda} \in Norm(\C O)$, 
then the $\mathcal{M}(\C O_{\lambda}, \epsilon)$ are precisely
the special unipotent representations $\mathcal{U}(\C O_{\lambda};\epsilon)$, which is irreducible.

\smallskip
Now consider any nilpotent orbit $\C P \in Norm(\C O)$. 
By induction hypothesis, assume that for all $\C P' \subsetneq \ovl{\C P}$, all the induced modules 
$\mathcal{M}(\C P', \epsilon') = \ovl{X}(\C P',\ep')$ are irreducible. 

Suppose on the contrary that $\Phi$ is a composition factor of $\mathcal{M}(\C P, \epsilon)$
other than $\ovl{X}(\C P, \ep)$, 
then $\Phi$ must have associated variety $\overline{\C P'} \subsetneq \overline{\C P}$ since 
$AV(\mathcal{M}(\C P, \epsilon)) = 1\cdot \ovl{\C P}$. Hence $\Phi = \ovl{X}(\C P', \epsilon')
= \mathcal{M}(\C P', \epsilon')$ for some $\epsilon'$, which has 
diminutive lowest $K_c-$type $V_{\omega}$. 
So we have
\begin{align*}
[\mathcal{M}(\C P, \epsilon): V_{\omega}] \geq \ [\ovl{X}(\C P, \ep):V_{\omega}] + [\Phi:V_{\omega}] > \ [\ovl{X}(\C P, \ep):V_{\omega}] \stackrel{\eqref{eq:last}}{=} &\ [\mathcal{M}(\C P, \epsilon):V_{\omega}],
\end{align*}
which gives a contradiction. So $\mathcal{M}(\C P, \epsilon) = \ovl{X}(\C P, \ep)$, and the result follows. \qed

\subsection{General Orbits} For general nilpotent orbits $\C O$ of Type $C$, the infinitesimal character $\la$ in \eqref{eq:infl} is not integral, and is formed of integers and
half-integers. The Kazhdan-Lusztig conjectures reduce this case to the
one considered in the previous section for an \textit{endoscopic
  group}. Let  
$$
\Delta(\la):=\{\al\in R\ : (\check\al,\al)\in \bZ\}.
$$
This root system is again of classical type. 
Then $\check\Delta(\la)=\{\check\al\ : \al\in\Delta(\la)\}$ forms a
root system. As in \cite{V6}, for instance, the character theory of $(\fk
g,K)-$modules at infinitesimal character $(\la,\la)$ can be deduced via
translation functors from a corresponding $(\la_{reg},\la_{reg})$ with the same
``\textit{integral roots}''. The category of representations with this
infinitesimal character decomposes into a sum of ``\textit{blocks}''
each with a \textit{coherent continuation action} of the Weyl group
$W(\la)$ generated by the root reflections coming from $\Delta(\la).$ 
The Kazhdan-Lusztig conjectures for non-integral infinitesimal
character state that the character theory of each block is equivalent
to the character theory of an \textit{endoscopic group} $G(\la)$ with roots
$\delta(\la).$ The induced module in Proposition
\ref{prop:parastructure} corresponding to an even nilpotent orbit. We
omit further details.

\ifx\undefined\bysame
\newcommand{\bysame}{\leavevmode\hbox to3em{\hrulefill}\,}
\fi

\end{document}